\documentclass[a4paper,11pt]{article}

\usepackage[T1]{fontenc}
\usepackage[english]{babel}
\usepackage[utf8]{inputenc}
\usepackage[a4paper]{geometry}

\usepackage{amssymb}
\usepackage{amsmath}
\usepackage{amsthm}
\usepackage{mathtools}

\usepackage{mathrsfs}
\usepackage{emptypage}
\usepackage{microtype}
\usepackage{enumerate} 
\usepackage{appendix}

\usepackage{color}
\usepackage[dvipsnames]{xcolor} 
\usepackage{tikz}
\usepackage{svg}

\usepackage{comment}
\usepackage{fixme}
\fxsetup{
	theme=color,
	author=,
}
\definecolor{fxnote}{rgb}{0.8,0,0}
\usepackage{hyperref}
\hypersetup{
	    colorlinks,
	    citecolor=black,
	    filecolor=black,
	    linkcolor=black,
	    urlcolor=blue,
	}
\usepackage{theoremref}

\newtheoremstyle{colorplain}%
{\topsep}   
{\topsep}   
{\itshape}  
{0pt}       
{} 
{.}         
{5pt plus 1pt minus 1pt} 
{\textbf{
	{\textbf{\thmname{#1} \thmnumber{#2}}}}\thmnote{ (#3)}}		
{}  

\newtheoremstyle{colordefinition}%
{\topsep}   
{\topsep}   
{}  
{0pt}       
{} 
{.}         
{5pt plus 1pt minus 1pt} 
{
	{\textbf{\thmname{#1} \thmnumber{#2}}}\thmnote{ (#3)}}
{}  

\newtheoremstyle{colorremark}%
{\topsep}   
{\topsep}   
{}  
{0pt}       
{\itshape} 
{.}         
{5pt plus 1pt minus 1pt} 
{
	{\thmname{#1} \thmnumber{#2}}\thmnote{ (#3)}}
{}

\theoremstyle{colorplain}			
\newtheorem{theorem}{Theorem}[section]
\newtheorem{lemma}[theorem]{Lemma}
\newtheorem{proposition}[theorem]{Proposition}
\newtheorem{corollary}[theorem]{Corollary}

\newtheorem{maintheorem}{Theorem}

\theoremstyle{colordefinition} 			
\newtheorem{definition}[theorem]{Definition}

\theoremstyle{colorremark}			

\newtheorem*{remark}{Remark}
\newtheorem*{example}{Example}

\newtheorem{conjecture}{Conjecture}

\usepackage[
	isbn=false,
	url=false,
	doi=false,
	eprint=false,
	style=alphabetic,
	giveninits,
	backend=biber]{biblatex}
\usepackage{csquotes}
\newbibmacro{string+doi}[1]{%
\iffieldundef{doi}{#1}{\href{http://dx.doi.org/\thefield{doi}}{#1}}}
\DeclareFieldFormat{title}{\usebibmacro{string+doi}{\mkbibemph{#1}}}
\DeclareFieldFormat[article]{title}{\usebibmacro{string+doi}{\mkbibquote{#1}}}
\renewcommand{\phi}{\varphi}
\renewcommand{\bar}{\overline}
\newcommand{\mysetminusD}{\hbox{\tikz{\draw[line width=0.6pt,line cap=round] (3pt,0) -- (0,6pt);}}}
\newcommand{\mysetminusS}{\hbox{\tikz{\draw[line width=0.45pt,line cap=round] (2pt,0) -- (0,4pt);}}}
\newcommand{\mysetminusSS}{\hbox{\tikz{\draw[line width=0.4pt,line cap=round] (1.5pt,0) -- (0,3pt);}}}
\newcommand{\Setminus}{\mathbin{\mathchoice{\mysetminusD}{\mysetminusD}{\mysetminusS}{\mysetminusSS}}}
\renewcommand{\setminus}{\Setminus}

\newcommand{\N}{\mathbb{N}}
\newcommand{\Z}{\mathbb{Z}}

\newcommand{\R}{\mathbb{R}}
\newcommand{\C}{\mathbb{C}}
\renewcommand{\H}{\mathbb{H}}
\renewcommand{\P}{\mathbb{P}}		

\newcommand{\CP}{\mathbb{CP}}

\newcommand{\de}{\partial}

\newcommand{\norm}[1]{\left\lVert#1\right\rVert}

\renewcommand{\O}{\mathcal{O}}		
\newcommand{\D}{\mathcal{D}}
\newcommand{\Hh}{\mathcal{H}}

\newcommand{\T}{\mathcal{T}}
\newcommand{\M}{\mathcal{M}}

\newcommand{\ML}{\mathcal{ML}}
\newcommand{\MF}{\mathcal{MF}}

\newcommand{\PT}{\mathcal{PT}}
\newcommand{\QT}{\mathcal{QT}}
\newcommand{\PM}{\mathcal{PM}}
\newcommand{\QM}{\mathcal{QM}}

\DeclareMathOperator{\id}{id}

\DeclareMathOperator{\Area}{Area}
\DeclareMathOperator{\Diff}{Diff}
\DeclareMathOperator{\Mod}{Mod}		
\DeclareMathOperator{\im}{Im}
\DeclareMathOperator{\re}{Re}
\DeclareMathOperator{\Len}{Len}

\DeclareMathOperator{\Gr}{Gr}
\DeclareMathOperator{\gr}{gr}

\DeclareMathOperator{\dev}{dev}
\DeclareMathOperator{\Infl}{Infl}
\DeclareMathOperator{\Sp}{Sp}
\DeclareMathOperator{\Met}{MMet}
\DeclareMathOperator{\met}{Met}
\DeclareMathOperator{\GL}{GL}

\DeclareMathOperator{\PSL}{PSL}

\DeclareMathOperator{\supp}{supp}
\DeclareMathOperator{\sys}{sys}
\DeclareMathOperator{\NPC}{NPC}

\title{A geometric boundary for the\\ moduli space of grafted surfaces}
\date{\vspace{-5ex}}
\author{Andrea Egidio Monti}

\begin{document}
\maketitle

\begin{abstract}	
Let $S$ be a closed orientable surface of genus at least two. We introduce a bordification of the moduli space $\mathcal{PT}(S)$ of complex projective structures, with a boundary consisting of projective classes of half-translation surfaces. Thurston established an equivalence between complex projective structures and hyperbolic surfaces grafted along a measured lamination, leading to a homeomorphism $\mathcal{PT}(S) \cong \mathcal{T}(S) \times \mathcal{ML}(S)$. Our bordification is geometric in the sense that convergence to points on the boundary corresponds to the geometric convergence of grafted surfaces to half-translation surfaces (up to rescaling). This result relies on recent work by Calderon and Farre on the orthogeodesic foliation construction. Finally, we introduce a change of perspective, viewing grafted surfaces as a deformation (which we term "inflation") of half-translation surfaces, consisting of inserting negatively curved regions.
\end{abstract}

\tableofcontents

\section*{Introduction}

Let $S$ be an orientable closed surface of genus $g$ at least 2.
Teichmüller space $\T(S)$ parametrizes all the complete hyperbolic metrics on $S$ up to isotopy, or equivalently all the marked Riemann surfaces diffeomorphic to $S$. Another family of well known geometric structures on $S$ is given by half-translation surfaces. They are surfaces obtained by taking a polygon in $\R^2$ and identifying its sides in pairs via the compositions of $\pm \id$ with translations.
A half-translation structure on $S$ is equivalent to the datum of a Riemann surface structure up to isotopy and a holomorphic quadratic differential on it, so the moduli space of such marked structure is described as $\QT(S)$ the bundle of quadratic differentials over $\T(S)$.

\paragraph{Grafting}

The main object of study in this article will be a third class of geometric structures that endow the surface $S$ with a piecewise hyperbolic and piecewise flat metric. Such surfaces are results of an operation called grafting.
In the simplest case grafting consists of the following: given a hyperbolic surface $X$, cutting it open along a simple closed geodesic and gluing in a flat cylinder in its place (see Figure \ref{fig:simple-grafting}). It can be then generalized to weighted multicurves, which are collections of disjoint simple closed geodesics, with a positive weight for each component, indicating the length of the cylinder to be inserted in place of each curve of such collection.

Grafting has been studied from various points of view and produced many results in the course of the years, since Thurston's seminal work until much more recent times.
The main result due to Thurston on grafting is its use to parametrize the moduli space of complex projective structures $\PT(S)$. 
Weighted multicurves sit inside a larger space $\ML(S)$ of measured laminations, that can be seen as a completion of the space of weighted multicurves.
As shown in Thurston's unpublished work and discussed in \cite{tanigawa97}, \cite{dumas09}, a complex projective structure induces a metric on the surface $S$, called the Thurston metric, that can be interpreted as the metric obtained by grafting a hyperbolic surface $X$ along a measured lamination $\lambda$, extending in this way the definition of grafting to measured laminations, and establishing a mapping class group equivariant homeomorphism
$$\Gr :  \T(S) \times \ML(S) \to \PT(S) .$$
For this reason we can think of $\PT(S)$ as parametrizing grafted surfaces and consider its points as surfaces with the grafted Thurston metric.

By looking at the underlying conformal structure of the grafted surfaces, that is by considering the composition $\gr = \pi \circ \Gr $ where $\pi: \PT(S) \to \T(S)$, fixing a lamination $\lambda$ induces \cite{SW02} a homeomorphism $\gr(\cdot, \lambda): \T(S) \to \T(S)$.
And on the other side, fixing a hyperbolic metric $X\in\T(S)$ induces \cite{DW08} a homeomorphism $\gr(X,\cdot ): \ML(S) \to \T(S)$, showing that grafting, like earthquakes, allows one to connect any two points of Teichmüller space with a so-called grafting ray $t \mapsto \gr(X,t\lambda)$.
Grafting rays have been studied from the point of view of their geometric properties.
It has been shown \cite{CDR10} that, under certain conditions, they fellow travel Teichmüller geodesics in $T(S)$. 

Since grafting is used to parametrize complex projective structures on a surface, it is then also intimately related to the theory of Kleinian groups.
Much more recently, it became apparent that this construction also plays a crucial role in higher Teichmüller theory, which makes it desirable to understand the moduli space of grafted surfaces from the geometric point of view. 
Indeed, in a recent joint work \cite{BHMM24}, grafting has proved to be useful to study the geometry of the moduli space of Hitchin representations. In that setting a notion of grafting along decorated multicurves is introduced as a deformation of Fuchsian representations into the Hitchin component.

This article relies on recent results \cite{CF21}, \cite{CF24} of Calderon and Farre.
In their substantial work they extend a construction by Thurston that was originally applicable only in the case of maximal laminations, that is for those laminations $\lambda$ whose complementary regions are ideal triangles.
In that case, given a hyperbolic surface $X$ and a maximal lamination $\lambda$, the horocycle foliations for the ideal triangles induce a measured foliation on the entire surface, transverse to $\lambda$.
In \cite{CF21} they introduce the orthogeodesic foliation $\O_\lambda(X)$, a generalization of the horocycle foliation that does not require $\lambda$ to be maximal and produces, in a mapping class group equivariant fashion, a measured foliation transverse to $\lambda$.
This allowed them to extend fundamental works like \cite{hubbardmasur79} and \cite{mirzakhani08} working only for maximal laminations, to the general case. In particular, they show that their construction of the orthogeodesic foliation induces a bijection $\O: \T(S) \times \ML(S) \to \QT(S)$ given by $\O(X,\lambda) = q(\O_\lambda(X), \lambda)$ where $q(\eta, \nu)$ indicates the quadratic differential associated to the pair $(\eta,\nu)$ of transverse measured foliations.

\paragraph{Main results}

The main goal of this work is to describe a bordification of the space $\PT(S)$ of grafted surfaces, where the boundary is described by a space of half-translation surfaces, that are for this scope considered up to isometry, that is up to the $S^1$ action given by rotations.
Our construction of such bordification follows the same idea of Thurston's compactification of $\T(S)$: we will define an ambient space $\Met(S)$ of marked metrics where both $\PT(S)$ and $\QT(S)/S^1$ embed, and then show the following.

\begin{maintheorem} \thlabel{mainth:bordification}
	Inside $\Met(S)$, $\PT(S)$ is asymptotic to the cone given by $\QT(S)/S^1$,
	or in other words, after projecting to the real projectification $\P\Met(S)$, we obtain 
	$$ \bar{\PT(S)} = \PT(S) \cup \P_\C\QT(S). $$
\end{maintheorem}

Such a bordification, being mapping class group equivariant, projects to the quotient and yields a bordification of the moduli space of (unmarked) grafted surfaces.

Other similar bordifications of $\PT(S)$ have been studied \cite{dumas06}, \cite{dumas07}. However, they differ from ours in the approach, as they are performed just by considering respectively the grafting and the Schwarzian parametrization of $\PT(S)$ and taking a bordification of the relative parameter spaces.
They result in a bordification that compactifies each fibre of $\PT(S) \to \T(S)$, that topologically correspond to the (real) projective compactification of $\R^{6g-6}$ into a disc. On the other side, in our bordification $\bar{\PT(S)}$ the closure of each fibre will be homeomorphic to the complex projective compactification of $\C^{3g-3}$.

However, the fundamental difference in our bordification is that 
points at infinity are geometric structures (half-translation surfaces) as well, and the convergence to such points is not the mere convergence in a parameter space, but real geometric convergence of the geometric structures.

The main argument of this article consists of showing a quantitative control on the convergence of a sequence of grated surfaces to a half-translation surface.
Such convergence happens in $\P\Met(S)$, meaning it is the convergence in $\Met(S)$ up to rescaling the metrics.
The topology we define on $\Met(S)$ is a topology  given by Lipschitz convergence compatibly with cone singularities.
We will use a lemma that shows that in our setting of surfaces with regular enough metrics, convergence  in the sense of Gromov-Hausdorff implies our notion of Lipschitz convergence.
So the key lemma to prove \thref{mainth:bordification} is the following.

\begin{maintheorem}	\thlabel{mainth:deflation-control}
	Let $\Gr(X,\lambda)$ be a grafted surface and $\O(X,\lambda)$ the associated half-translation surface. Consider them to be rescaled to have unit area.
	If the injectivity radius of $\O(X,\lambda)$ is at least $\epsilon$ and $\ell_X(\lambda) > 1$, then there exists a constant $C$ depending only on $\epsilon$ and $S$ such that 
	$$ d_{GH}(\Gr(X,\lambda), \O(X,\lambda)) \leq C \cdot \left(\ell_X(\lambda)\right)^{-1/2}. $$
\end{maintheorem}

As mentioned above, our bordification $\bar{\PT(S)}$ brings together in a single moduli space hyperbolic, grafted and flat half-translation surfaces.
Hyperbolic structures are indeed naturally contained in $\PT(S)$ as the grafted surfaces obtained with trivial grafting along the zero lamination.
According to Thurston's work, properly grafted surfaces - the ones obtained with non-trivial grafting - are by definition deformations of hyperbolic ones.
We present a dual point of view, by giving a parametrization of half-translation surfaces together with properly grafted surfaces as a deformation of the first.
We call such deformation \emph{inflation}, taking inspiration for the name from the deflation map, introduced in \cite{CF21}, that goes vice versa from a grafted surface $\Gr(X,\lambda)$ to the half-translation surface $\O(X,\lambda)$.

\begin{maintheorem} \thlabel{mainth:inflation}
	Let us consider the bordification of $\PT(S)$ described above, and the projective bordification $\bar{\QT(S)} = \QT(S) \cup \P_\C\QT(S)$.
	By extending the map $\O^{-1}: \QT(S) \to \PT(S)$ to the boundary as the identity on $\P_\C\QT(S)$ we get a map
	$$\Infl: \bar{ \QT(S)} \to \bar{\PT(S)}$$
	that is a bijection between the domain and the complement of $\T(S)$ in $\bar{\PT(S)}$, and it is continuous on the boundary $\P_\C\QT(S)$.
\end{maintheorem}

Here is the geometric interpretation of inflation and its analogy with grafting.
For grafting, a pair $(X,\lambda)$ is provided. When $\lambda$ is the 0 lamination, the grafting returns the hyperbolic surface $X$. When $\lambda$ is non-zero, grafting consists in inserting flat structures along the leaves of $\lambda$.
For inflation, a quadratic differential $q\in\bar{\QT(S)}$ is provided. When it is on the boundary, so only its projective class is specified, inflation returns the unit area half-translation surface associated to $q$. When $q$ has a finite norm $|q|$, then inflation consists in inflating the diagram of singular leaves for the vertical foliation associated to $q$ into a fat graph endowed with negative curvature $-|q|$.
The bigger the norm $|q|$, the more negatively curved and thinner the inflated fat graph will be (see \thref{th:slim-inflation}), with the limit of the fat graph, for $|q|$ going to infinity, being the diagram of singular leaves itself, making the inflation trivial. Such collapsing phenomenon of a single "inflated" component was similarly observed also in \cite{xu19}.

Analogously to grafting rays, we can define in this way inflation rays $t \mapsto \Infl(e^t q)$. It follows from the definition of the map $\O$ that when composed with the projection $\PT(S) \to \T(S)$ they trace paths that generalize stretch lines \cite{thurston86}. As discussed in \cite[Section 15]{CF21} it is not yet clear whether the so-defined generalized stretch lines always describe Thurston geodesics in $\T(S)$, as it happens \cite{papadopoulos21}, \cite{disarlo22} for different cases where a certain degree of symmetry for the stretch lamination is present.

\paragraph{Further work and future directions}
In a further work in preparation we investigate another bordification, equivalent to the one described in the above-mentioned work of Dumas \cite{dumas06}, that arises instead from the natural mapping of $\T(S)\times\ML(S)$ into the space of geodesic currents, given just by taking the sum of the Liouville current and the measured lamination. 
The main idea is to use again geometric convergence, in the sense of Gromov-Hausdorff, of grafted surfaces to half-translation surfaces,
where this time the associated metrics are not the classical ones, but an $L^1$ variation of those.
From this it seems possible to deduce some of the results in \cite{theret07}, \cite{theret18} on the asymptotic behaviour of length functions along stretch lines, and extend them to the general case where no further hypothesis are required for the measured lamination.

The bordification $\bar{\PT(S)}$ we described has the advantage of being characterized by geometric convergence of surfaces with a metric, and hence offers a tool to compute in general limits of geometric invariants, such as length functions and volume entropy, for sequences of grafted surfaces converging to points on the boundary. 

Due to the geometric nature of all our construction, all the results we stated pass to the quotient by the mapping class group and are valid for the non-marked version of the moduli spaces considered.
It would be interesting to extend in this context our bordification to a compactification, meaning considering also Gromov-Hausdorff limits of sequences of degenerating surfaces, where the limit is not topologically a surface.

\textbf{Conjecture:} A pointed Gromov-Hausdorff limit of a sequence of area-normalized grafted surfaces is either a finitely grafted nodal hyperbolic surface, or the pointed Gromov-Hausdorff limit of a sequence of half-translation surfaces.

\paragraph{Outline of the paper}
In Section \ref{sec:preface} we introduce the classical objects in Teichmüller theory of hyperbolic and half-translation surfaces, and recall some standard results.
In Section \ref{sec:grafting} we introduce grafting and present Thurston's parametrization of the space of projective structures using grafting.
In Section \ref{sec:results} we give the definition of the space $\Met(S)$, encompassing both grafted and half-translation surfaces, and state \thref{mainth:bordification} in this setting.
Section \ref{sec:degrafting} is about the quantitative convergence in $\Met(S)$ for sequences of grafted surfaces where the measure of the lamination goes to zero. It is independent of the main result, but introduces in an easier setting the same ideas of the main argument of the article.
Section \ref{sec:orthogeodesic} contains the summary of the results that we import from \cite{CF21} and that are fundamental for our argument. We go through and readapt a construction of \cite{CF21} in order to define our own version of the deflation map $\D: \Gr(X,\lambda) \to \O(X,\lambda)$, laying the foundations for the main argument.
Section \ref{sec:deflation} contains the proofs of \thref{mainth:bordification} and \thref{mainth:deflation-control}, showing first that provided that $\ell_X(\lambda)$ is large, then the negatively curved part of $\Gr(X,\lambda)$ area-normalized is slim, and then that $\D$ is an $\epsilon$-isometry. 
Finally, in Section \ref{sec:inflation} we introduce inflation, try to provide a geometric intuition for it, and show its relation with stretch lines.

\paragraph{Acknowledgment}
The author is grateful to his supervisor U. Hamenstädt for the helpful conversations and guidance on this work.
The author is supported by the Max Planck Institute for Mathematics in Bonn.

\section{Preface} \label{sec:preface}

	We introduce here the classical objects and structures we will use in this article. All the statements of this section are well known facts in Teichmüller theory and can be found in introductory books such as \cite{imayoshi}, \cite{farbmarg}, \cite{martelli}.

	\subsection{Teichmüller space}
		Let us assume for the entire article that $S$ is a closed orientable surface of genus $g\geq 2$. It is known that $S$ admits many hyperbolic structures.
		Let us denote by $\T(S)$ the Teichmüller space of the surface $S$, the space where to each point corresponds an isotopy class of a complete hyperbolic metric on $S$. 
		We remind that thanks to the correspondence between hyperbolic structures and complex structures, $\T(S)$ can be thought of parametrizing either the hyperbolic structures on $S$ up to isotopy, or the Riemann surface structures up to isotopy.
		Equivalently, a point in $\T(S)$ can be also described as a pair $(X,[f])$ where $X$ is a hyperbolic surface and $[f]$ the homotopy class of $f: S \to X$ a diffeomorphism called \emph{marking}. Together they determine by pull-back a hyperbolic structure on $S$, unique up to isotopy.

		The mapping class group $\Mod(S) = \Diff^+(S)/\Diff^+_0(S)$ of the surface $S$ acts on $\T(S)$ by pre-composition with the marking, so the quotient $\M(S) = \T(S)/\Mod(S)$ is the moduli space of hyperbolic surfaces up to isometry (or equivalently the space of Riemann surfaces up to biholomorphism) and it is commonly referred to just as \emph{moduli space} of $S$.

		The moduli space is not compact as there exist sequences of hyperbolic surfaces with a closed geodesic whose length goes to zero.
		Given any $\epsilon>0$ the subset of $\M(S)$ of surfaces with no closed geodesic of length smaller than $\epsilon$ is compact, and its preimage $\T_\epsilon(S)$ in $\T(S)$ is called \emph{$\epsilon$-thick part} of Teichmüller space.

		\subsection{Half-translation surfaces and quadratic differentials}

		A half-translation structure on a surface $S$ is an atlas for $S\setminus F$, where $S\subset F$ is finite, made of charts to $\R^2$ such that transition maps are $\pm \id$ composed with a translation.
		Such a structure induces a flat singular metric on $S$, with finitely many singular points, each with cone angle $k\pi$ for some integer $k>2$.
		The analogue of Teichmüller space for half-translation structures will be denoted with $\QT(S)$.

		The space $\QT(S)$ is parametrized by holomorphic quadratic differentials. A holomorphic quadratic differential $q$ on a Riemann surface $X$ is an object that in holomorphic charts reads as $q(z)=\phi(z)dz^2$ with $\phi(z)$ holomorphic. The holomorphic quadratic differentials on $X$ form a vector space $Q(X)$ of complex dimension $3g-3$.
		If we consider charts of a half-translation surface structure from subsets of $S\setminus F$ to $\R^2 \cong \C$, these induce a complex structure on $S\setminus F$ that one can show extends to the whole surface $S$. Moreover, the quadratic differential expressed as $dz^2$ in such charts also extends to a holomorphic quadratic differential on the whole surface.
		Such correspondence identifies $\QT(S)$ with the vector bundle of quadratic differentials over Teichmüller space. We will make no distinction when talking about a half-translation surface and its associated quadratic differential $q$.

	\subsection{Measured laminations}
		Geodesic laminations are interesting objects that have been very useful in many ways in Teichmüller theory.
		A \emph{geodesic lamination} $\lambda$ on a hyperbolic surface $X$ is a compact subset foliated by geodesics.
		Some geodesic laminations admit a transverse measure, that is the datum of inducing of a measure on each transverse arc $\alpha$, such that the support is the entire intersection $\alpha \cap \lambda$, and such that it is invariant under homotopies of $\alpha$ preserving its transversality. We will indicate with $i(\lambda, \alpha)$ the total mass of the transverse measure induced on the arc $\alpha$.
		The laminations together with such a transverse measure are called \emph{measured laminations}, and the geodesic lamination underlying a measured lamination is called its \emph{support}.

		Measured laminations can have both closed and infinite leaves. When a measured lamination has only closed leaves, its support is a \emph{multicurve}, that is a finite union of simple closed geodesics, pairwise non-homotopic. In this case, the transverse measure is determined by a positive weight for each component, and the measured lamination is called \emph{weighted multicurve}.
		
		The measure of the intersection of laminations with suitable sets of curves gives charts that induce a piecewise linear structure on the space $\ML(S)$ of measured laminations, and, in particular, a topology that makes it homeomorphic to $\R^{6g-6}$.
		According to this topology the set of weighted multicurves is dense in $\ML(S)$.
		In general when we consider a measured lamination we would assume that it is non-empty; although in some occasion, we will need to consider also the zero lamination. For this reason we will specify with the notation $\ML_0(S)$ when we mean to include the zero lamination.
		
		Geodesic laminations are topological objects, as upon changing the underlying hyperbolic metric on $S$, it is always possible to straighten them to be geodesic according to the new metric and the topology induced on the space $\ML(S)$ does not change.
		
		There is a continuous function $\ell: \T(S)\times \ML(S) \to \R_{\geq 0}$ called \emph{length function}, and that we indicate as $\ell_X(\lambda)$ with the following properties. For $\lambda$ a unit mass closed geodesic, $\ell_X(\lambda)$ is the length of the geodesic representative of $\lambda$ according to the hyperbolic metric $X$; it is linear in $\lambda$; and additive for disjoint laminations.

		\paragraph*{Train tracks}
		Train tracks are very convenient combinatorial object used to study measured laminations. An extensive study of these can be found in \cite{penner91}.
		A train track $\tau$ on $S$ is an embedded graph on $S$, where the edges are called \emph{branches} and the vertices \emph{switches}, and at all switches all the adjacent branches are parallel to a single direction and there is at least one incoming branch for each of the two orientations.
		For every lamination $\lambda$ there is a train track $\tau$ that \emph{carries} it, meaning that $\lambda$ can be homotoped onto it (see geometric train track in Section \ref{sec:construction}). The transverse measure of $\lambda$ induces a transverse measure on the branches, that is simply encoded by a \emph{weight system}, that is the assignation of a non-negative weight for each branch, such that the switch condition is satisfied. The switch condition requires that at every switch the sums of the weights of the branches on each of the sides are equal.
		Vice versa, a train track $\tau$ with a weight system always encodes a measured lamination carried by $\tau$. If all the weights are positive, we say that $\tau$ carries snugly $\lambda$.

	\subsection{Measured foliations}
		To a half-translation surface structure $q \in \QT(S)$  is associated a pair of singular foliations on $S$, $(\re(q), \im(q))$, called respectively \emph{vertical} and \emph{horizontal foliation}, as they are defined as the two measured foliations whose leaves are respectively vertical and horizontal in the charts given by $q$. The transverse measure is the one given by the Euclidean metric in the horizontal and vertical directions respectively.

		There is a canonical identification (see \cite{levitt83}) between the space of measured laminations $\ML(S)$ and the space $\MF(S)$ of measured foliations on $S$ up to homotopy and Whitehead moves (operations that do not change essentially the lamination, except for merging or separating two singularities connected by a leaf).

		The association of the pair of vertical and horizontal foliations $(\re(q), \im(q))$ to the half-translation surface $q$ induces an embedding $\QT(S) \hookrightarrow \MF(S) \times \MF(S)$. The image of the embedding is given by all the pairs of foliations that can be realized as transverse foliations (meaning with same singular points, and transverse everywhere else). Let us denote by $q(\eta,\lambda)$ the quadratic differential associated to the pair of measured foliations $(\eta,\lambda)$. 
		Given $\lambda \in \MF(S)$, we denote by $\MF(\lambda)$ the space of measured foliations transverse to $\lambda$.
		It is known, from the work of Hubbard and Masur \cite{hubbardmasur79}, that $\MF(\lambda)$ is homeomorphic to $\T(S)$.

\section{Grafting and projective structures} \label{sec:grafting}

	In this section we will introduce grafting first as a geometric operation to produce a surface with piecewise hyperbolic and piecewise flat metric. Then we will introduce complex projective structures on surfaces, and show how grafting can be described as a deformation of such structures. We briefly present Thurston's result on the parametrization of the moduli space of complex projective structure using grafting, and show the equivalence between the projective and geometric points of view. We finally introduce a space of metrics on $S$ where the space of grafted surfaces embeds.
	
	\subsection{Geometric grafting}

		Let us first see the easiest example, also called \emph{simple grafting} where the inserted flat part is only one cylinder. Later on we will extend the definition to a more general setting.
		Let $X$ be a hyperbolic surface, $\gamma$ a simple closed geodesic in it and $a> 0$ a real number. We will perform a grafting of $X$ along $\gamma$ with weight $a$.
		For easiness of notation, here as in the rest of the paper, we will use the same symbol to indicate both the parametrization and the image of curves and arcs.

		\begin{definition}
			Let us consider a cylinder $\gamma \times [0,a]$ endowed with the Euclidean product metric, so that it has circumference of length $\ell_X(\gamma)$ and height $a$.
			Let us cut $X$ open along $\gamma$, obtaining a hyperbolic surface $X'$ with two geodesic boundary components $\gamma_1, \gamma_2$, each of length $\ell_X(\gamma)$, with parametrizations compatible with the cut,
			meaning that for every $t$, the points $\gamma_1(t), \gamma_2(t)$ both correspond to $\gamma(t)$ before the cut.
			Then glue the euclidean cylinder to $X'$ along the boundary with the isometries $\phi_1: \gamma \times \{0\} \to \gamma_1$, and $\phi_2: \gamma \times \{a\} \to \gamma_2$. The result is called \emph{grafted surface} $\Gr_{a\gamma}(X)$.
		\end{definition}

		\begin{figure}[ht]
			\includesvg[width=0.8\textwidth]{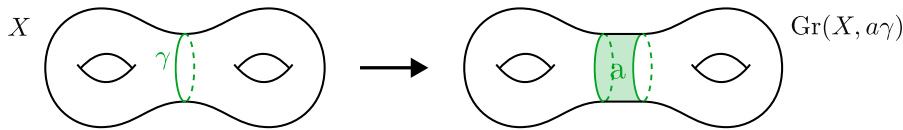}
			\centering
			\vspace{0.2cm}
			\begin{minipage}{0.8\textwidth}
				\caption{Simple grafting along $\gamma$ with weight $a$.}
				\label{fig:simple-grafting}
			\end{minipage}
		\end{figure}

		Note that the obtained grafted surface has a well-defined differentiable structure: the gluing is performed along the geodesic boundary of $X'$ and the geodesic boundary of the cylinder.
		Since the gluing is performed via isometries the metric of the grafted surface is Riemannian with $C^1$ regularity, but not $C^2$, as we notice that the curvature is discontinuous, as it jumps from locally constant $-1$ in the hyperbolic part to locally constant 0 in the flat cylinder.

		There is a natural homotopy equivalence from $\Gr_{a\gamma}(X)$ to $X$ that consists in collapsing the grafted cylinder to the curve $\gamma$. This homotopy equivalence allows us, given a marking $f: S \to X$ to induce one on $\Gr_{a\gamma}(X)$, so grafting can be also seen as an operation on marked structures.
		
		Moreover, we observe that the grafting datum is only topological, as the curve $\gamma$ can be seen just as a free homotopy class of simple closed curves on $S$. In order to perform the grafting on $X$ we just need to push forward $\gamma$ to $X$ via the marking $f$, take there the closed geodesic homotopic to the push forward of $\gamma$, and perform the grafting operation along it.
		So grafting is well-defined on the pairs $(X,a\gamma)$ where $X\in\T(S)$ and $a\gamma$ is a weighted free homotopy class of simple closed curves on $S$.
		
		Grafting easily extends to weighted multicurves.
		So given $X \in \T(S)$ and $\mu = a_1 \gamma_1 + \dots + a_k \gamma_k$, the grafted surface $\Gr_\mu(X)$ is obtained by inserting one cylinder of height $a_i$ for each component $\gamma_i$ of $\mu$.

		We call \emph{hyperbolic pieces} the connected components of the complement of the cylinders in $\Gr_{a\gamma}(X)$. They are the interior of hyperbolic surfaces with geodesic boundary.

	\subsection{Complex projective structures}

		Let us consider the Riemann sphere $\CP^1$. The projective automorphisms of $\CP^1$ are the maps induced by the tautological linear action of $\GL_2\C$ on $\C^2$, that on $\CP^1$ acts with kernel $\C^* \id$. So the automorphisms group is $\PSL_2\C = \GL_2\C / \C^*\id$, and its elements are also known as \emph{Möbius transformations}.

		\begin{definition}
			A complex projective structure on $S$ is a maximal atlas for $S$ of charts to $\CP^1$, where the transition maps are restrictions of Möbius transformations.
		\end{definition}
		We will often omit the word \emph{complex} and with \emph{projective structures} we always mean the complex projective ones. A map between surfaces with projective structures is called \emph{projective} if in charts it is read as the restriction of a Möbius transformation. 

		The definition of projective structures is a particular case (with $X=\CP^1$ and $G=\PSL_2\C$) of the more general concept of $(G,X)$-structure on a smooth $n$-dimensional manifold, where $X$ is a real analytic manifold of dimension $n$, and $G$ is a Lie group acting faithfully and analytically on $X$. The theory of $(G,X)$-structures is developed in \cite[Chapter I.1]{notesonnotes}.
		We will state here some properties of projective structures, that are commonly known for hyperbolic structures (which are also $(\PSL_2\R, \H^2)$-structures), and true in general for any $(G,X)$-structure.

		Given a projective structure $Z$, we recall that the universal cover $\tilde S$ of $S$ also inherits naturally from $Z$ a projective structure just by pre-composition of the atlas for $S$ with the covering map. We call $\tilde Z$ the surface $\tilde S$ with such an inherited projective structure.

		\begin{proposition}
			For every projective structure $Z$, there exists a homomorphism $\rho_Z: \pi_1(S) \to \PSL_2\C$ called \emph{holonomy representation}, and a $\rho_Z$-equivariant projective map $\dev_Z: \tilde Z \to \CP^1$ called \emph{developing map}.
			The pair $(\rho_Z, \dev_Z)$ is unique, up to the action of Möbius transformations. This action of $\phi\in\PSL_2\C$ is given by conjugation on $\rho_Z$ and post-composition on $\dev_Z$ with the lift $\tilde \phi \in \Diff_0^+(\tilde S)$.
		\end{proposition}

		\begin{definition}
			The moduli space $\PT(S)$ of marked projective structures on $S$, similarly to the case of marked hyperbolic structures, can equivalently be defined as:
			\begin{enumerate}
				\item the set of projective structures on $S$ up to isotopy
				\item the set of pairs of developing map and holonomy, up to isotopy
				$$\PT(S) = \{ (\dev, \rho) \;|\; \rho: \pi_1(S) \to \PSL_2\C, \, \dev: \tilde S \to \CP^1 \text{ $\rho$-equiv. } \} / \Diff_0^+(S) $$
				\item the set of marked projective surfaces
				$$ \PT(S) = \left\{ (Z,[f]) \;|\; Z \text{ projective surface}, \; f:S \to Z \text{ marking} \right\}/ \sim $$
				where $(Z,[f]) \sim (Z',[f'])$ if there exists a projective diffeomorphism $\phi: Z \to Z'$ such that $\phi \circ f $ is homotopic to $f'$.
			\end{enumerate}
		\end{definition}

		The last one also allows us to endow $\PT(S)$ with a topology. Indeed, the set of all possible developing maps from $\tilde S$ to $\CP^1$ has naturally the compact open topology, or in other words, the topology given by the uniform convergence on compact subsets. This is an infinite dimensional space, but we can take the quotient by the action of the infinite dimensional Lie group $\Diff_0^+(S)$ and so endow the space $\PT(S)$ with the quotient topology.

		\begin{remark}
			A complex projective structure induces also a complex structure, as the Möbius transformations are holomorphic. Thus, we can define a forgetful map $\pi: \PT(S) \to \T(S)$ that remembers only the complex structure.
		\end{remark}

		There is a parametrization of $\PT(S)$, called Schwarzian parametrization (see \cite{dumas07}), that gives a homeomorphism $\PT(S) \to \QT(S)$ compatible with forgetful maps from the spaces to $\T(S)$. This will imply for example that $\PT(S)$ is finite dimensional. Although, we will not use this parametrization, but rather a more geometrical one using grafting and introduced by Thurston, which also have some additional desirable properties, as for instance it is canonical.

		\begin{remark}
			We observe that $\T(S)$ embeds naturally into $\PT(S)$. Indeed,
			the hyperbolic plane $\H^2$ can be modelled as the upper half plane $\Hh$ of $\C\subset \CP^1$, with isometry group given by the real Möbius transformations $\PSL_2\R < \PSL_2\C$. So the developing map and holonomy of a hyperbolic structure are also developing map and holonomy for a projective structure.
		\end{remark}

	\subsection{Projective grafting}
		As we pointed out in the previous section, Teichmüller space is naturally embedded in $\PT(S)$, meaning that every hyperbolic surface induces a projective structure too. We will see now how grafting can be seen also as a deformation of such projective structures, and see how Thurston used it to parametrize the space $\PT(S)$.

		For this scope, let us introduce the following model for cylinders with a projective structure.  
		Consider, for $a<2\pi$, $V_a$ the sector of the complex plane defined as 
		$$ V_a = \{ r e^{i\alpha} \,\mid\, r>0,\, \alpha\in (0,a) \}. $$
		It has a natural projective structure given by the inclusion in $\C \subset \CP^1$. The quotient $V_{a,l}$ of $V_a$ by the $\Z$-action generated by the Möbius transformation $z \mapsto e^l z$, with $l\in\R^+$, is a cylinder endowed with the quotient projective structure. 
		For $a>2\pi$ we can construct $V_{a,l}$ as a concatenation, that is by gluing along their boundaries $n$ copies of projective cylinders $V_{a/n, l}$ for $n$ large enough so that $a/n < 2\pi$.
		Now we can define projective grafting totally analogously to how we defined the geometric one.

		\begin{definition}
			
			Given $X\in\T(S)$ and $\mu = \sum_i a_i\gamma_i$ a weighted multicurve, we cut $X$ along the geodesic representative of $\mu$ and for every term $a_i\gamma_i$ we insert a projective cylinder $V_{a_i, \ell_X(\gamma_i)}$ in place of the closed geodesic $\gamma_i$.
			It is easy to verify, by comparing the developing maps, that the projective structure on the hyperbolic pieces and the cylinders match nicely and form a projective structure on the whole surface obtained. 
		\end{definition}

		This operation defines a grafting map $\Gr: \T(S) \times \mathscr{M}(S) \to \PT(S)$ where $\mathscr{M}(S)$ is the set of weighted multicurves. This leads to Thurston parametrization theorem of $\PT(S)$ using grafting.
		
		\begin{theorem}[Thurston]	\thlabel{th:thurston-param}
			The projective grafting defined above extends to a mapping class group equivariant homeomorphism
			$$ \Gr: \T(S) \times \ML_0(S) \to \PT(S). $$
		\end{theorem}
		
		This means that grafting completely parametrizes the space of projective structures. And vice versa, for every grafting datum $(X,\lambda) \in \T(S) \times \ML_0(S)$ there is a projective structure $\Gr(X,\lambda)$ associated to it. The goal of the next section is to view $\Gr(X,\lambda)$ as geometric object, and try to understand its geometry also in the case when $\lambda$ is not a weighted multicurve.

	\subsection{Thurston metric}		\label{sec:thurston-param}

		In this section we explain why grafting and projective grafting are essentially the same. Thurston showed that any projective structure $\Gr(X,\lambda)$ on $S$ induces a metric that, in the case of a structure obtained by projective grafting along a multicurve, coincides with the geometrically grafted metric $\Gr_\lambda(X)$ of the section above.
		The definition of such a metric given by Thurston, valid in the general case, is an adaptation to the projective case of the Kobayashi metric.

		\begin{definition}[Thurston metric]
			Given $Z\in \PT(S)$, the Thurston metric associated to it is a norm on the tangent bundle of $S$ defined as follows. For every $p\in S$ and $v \in T_pS$
			$$ \norm{v}_{Th} = \inf_{j} \norm{ j^*v }_{hyp} $$
			where the infimum is taken over all projective immersions $j: \Delta \to Z$ from the complex unit disc $\Delta$, containing $p$ in their image, and where $\norm{\cdot}_{hyp}$ is the hyperbolic metric on $\Delta$.
		\end{definition}
		
		We recall that the complex unit disc $\Delta = \{z\in \C | z\bar z <1\} $ can be both seen as the Poincaré disc model for the hyperbolic plane $\H^2$ and as $\Delta \subset \C \subset \CP^1$ inside the Riemann sphere, which gives it a canonical projective structure.

		A discussion on the Thurston metric, in a more general setting, can be found in \cite{kulkarni94}, where in particular they show the following.

		\begin{theorem}[{\cite{kulkarni94}}] \thlabel{th:kulkarni}
			Let $Z$ be a surface with a complex projective structure. The associated Thurston metric is a complete Riemannian metric $g_Z$ of class $C^{1,1} $ (i.e. with Lipschitz derivatives), compatible with $Z$, i.e. it is conformal to $\pi(Z)$.
			Moreover, the Lipschitz constant of its derivatives is bounded locally in $\PT(S)$.
		\end{theorem}

		Our goal is now to understand better the geometry of projective surfaces with the Thurston metric. For this scope we need to extrapolate some results from the proof of \thref{th:thurston-param}.
		The idea of Thurston's proof  goes in the opposite direction of the grafting map, showing a construction for retrieving, given a projective structure $Z$, the grafting datum $(X,\lambda)$ and proving that this depends continuously on $Z$.
		At the core of the construction is producing a map $\kappa: Z \to X$ that will play the role of the generalization of the collapsing map we mentioned in the previous section.
		A more detailed discussion of Thurston's argument can be found in \cite{dumas09} and \cite{tanigawa97}.
		Here we summarize in the following statement only some particular geometric information that emerges from Thurston's argument.

		\begin{theorem}[Thurston]
			Given a projective structure $\Gr(X,\lambda) \in\PT(S)$, there exists a homotopy equivalence $\kappa: \Gr(X,\lambda) \to X$ such that
			\begin{enumerate}[i.]
				\item $\kappa^{-1}(\lambda)$ is a lamination on $\Gr(X,\lambda)$, geodesic with respect to the Thurston metric
				\item $\kappa$ maps each geodesic in $\kappa^{-1}(\lambda)$ isometrically to a leaf of $\lambda$
				\item for every connected component $A$ of $X\setminus \lambda$, $\kappa$ restricts to an isometry between $\kappa^{-1}(A)$ and $A$. 
			\end{enumerate}
			Moreover, in the case when $\lambda$ is a weighted multicurve, the surface $\Gr(X,\lambda)$ equipped with the Thurston metric coincides with the geometrically grafted surface $\Gr_\lambda(X)$ obtained from $X$ inserting Euclidean cylinders along $\lambda$, and the map $\kappa: \Gr(X,\lambda) \to X$ is the map collapsing the cylinders.
		\end{theorem}

		For this reason from now on, we will unify the notation, and with $\Gr(X,\lambda)$ we will mean a surface equipped with both the projective structure and the Thurston metric. 
		In the general case, when $\lambda$ is not a weighted multicurve, the behaviour of the map $\kappa$, which we will still call \emph{collapsing map}, is more complicated. See for example \ref{sec:cantor} for a better understanding of the local behaviour.
		Nevertheless, the following result holds true.

		\begin{corollary}	\thlabel{th:kappa-lip}
			The collapsing map $\kappa: \Gr(X,\lambda) \to X$ is 1-Lipschitz.
		\end{corollary}
		\begin{proof}
			It is 1-Lipschitz in the case when $\lambda$ is a weighted multicurve as it is piecewise a local isometry or the collapse map along a cylinder. The Thurston metric of $\Gr(X,\lambda)$ depends continuously on the grafting datum, and so does the collapsing map $\kappa$, so by continuity $\kappa$ is 1-Lipschitz for any $\lambda$ measured lamination.
		\end{proof}

\subsection{Continuity of grafting}

	According to the definition, each point in $\PT(S)$ is an equivalence class of projective structures, that is it determines a projective structure, and together with it the associated Thurston metric, both only up to the action of elements of $\Diff_0^+(S)$.
	In order to overcome this ambiguity we have the following.
	
	\begin{lemma}
		The projection to the quotient by the action of $\Diff_0^+(S)$ from the space of all projective structures on $S$ to $\PT(S)$ admits a continuous section.
	\end{lemma}
	\begin{proof}
		For every $Z\in\PT(S)$ we need to pick a representative, and make sure the choice depends continuously on $Z$.
		We will follow the same idea of \cite{wolf89}, where harmonic maps are used to relate any hyperbolic surface with a fixed one, in order to identify $\T(S)$ with a precise set of Riemannian metrics on a fixed surface. 

		More precisely, let us fix a hyperbolic structure $X\in\T(S)$ on $S$.
		Then given a projective structure $Z$ on $S$ we choose its representative $\hat Z$ with respect the action of $\Diff_0^+(S)$ such that the identity map on $S$ is harmonic between $X$ and $\pi(\hat Z) \in \T(S)$, where here $\pi(\hat Z)$ is meant as the hyperbolic metric conformal to the complex structure induced by $Z$. This can be achieved by simply choosing $\hat Z$
		to be the pull-back of $Z$ via the harmonic map between $X$ and $\pi(Z)$.

		One can check that this section $Z \mapsto \hat Z$ is continuous with respect to the topology on $\PT(S)$ as the forgetful map $\pi: \PT(S) \to \T(S)$ is continuos and the harmonic map depends continuously on the target metric $\pi(Z)$.
		
	\end{proof}
	
	From now on we will imply the choice of a fixed section and by a point $Z \in \PT(S)$ we mean a precise projective structure on the surface $S$.
	The associated Thurston metric $g_Z$ can be then considered as a metric on the very same surface $S$. For this reason by $Z$ we will mean the surface $S$ equipped with the Thurston metric $g_Z$.
	In this way it will make sense for example to consider points $x,y$ on $S$ and compare their distances $d_{Z_n}(x,y)$ with respect to the metrics induced by different structures $Z$.
	
	We now have the following result, essentially a consequence of the regularity result in \thref{th:kulkarni}, and that can be found in \cite[Section 4.3]{dumas09}. 

	\begin{theorem}	\thlabel{th:grafting-continuity}
		The Thurston metric $g_Z$ on $S$ depends continuously on the projective structure $Z\in\PT(S)$, with respect to the $C^0$ topology on the space of $C^1$ Riemannian metrics on $S$.
	\end{theorem}

\section{Moduli space of grafted surfaces}		\label{sec:results}

	In this section we present the main result of the article and so define a bordification $\bar{\PT(S)}$ of the space of marked grafted surfaces, where the boundary is given by marked half-translation surfaces. What we mean by bordification of a space is a larger space that contains the former as a dense subspace equipped with the subspace topology. When a bordification is compact, it is a compactification. Our bordification $\bar{\PT(S)}$ will not be compact, not even after taking the quotient by the action of the mapping class group.
	
	The construction we perform in this article is guided by the following conjecture we vaguely state.
	The idea takes inspiration from the construction of the Thurston's compactification of Teichmüller space.

	\begin{conjecture}
		Let $\NPC(S)$ be the moduli space of non-positively curved marked surfaces,
		where elements are pairs $(X,[f])$ where $f:S \to X$ is a marking and $X$ is a surface equipped with a metric with adequate hypothesis, and the pairs are considered up to isometries compatible with the markings.
		Then the natural inclusion of $\PT(S)$ in $\NPC(S)$ is a proper embedding. 
		By taking the projectivisation, $\PT(S)$ embeds also in $\P\NPC(S)$ and the closure of its image here is
		$$ \bar{\PT(S)} = \PT(S) \cup \P_\C\QT(S)$$
		where $\P_\C\QT(S)$ represents the class of marked metrics coming from half-translation surfaces (see details below).
	\end{conjecture}

	We believe that the conjecture above holds true for a rather general definition of $\NPC(S)$ with a suitable topology.
	However, since the geometric structures we are interested in studying in this article have a very explicitly description, we avoid introducing non-positive curved metrics in a wider generality and only consider a moduli space for grafted and half-translation surfaces.

	\begin{definition}
		Let us define $\Met(S)$ the space of marked surfaces coming either from rescaling a grafted surface or from a half-translation surface
		$$ \Met(S) = \{ (X,[f]) \; \mid \; X \text{ as above, } f:S\to X \text{ marking} \}/\sim $$
		where, as usual, we say $(X,[f]) \sim (X', [f'])$ if there is an isometry $\phi:X \to X'$ compatible with the markings, that is $[\phi \circ f]=[f']$ as homotopy classes.
	\end{definition}

	Notice that we allow grafted surfaces to be rescaled, so that the whole space $\Met(S)$ supports a $\R^+$ action by rescaling the metrics.
	Let us now define a topology on the space $\Met(S)$.

	\begin{definition}
		Given a marked surface $(X,[f]) \in \Met(S)$, we define a collection of its neighbourhoods $V(\epsilon, U)$ for $\epsilon >0$ and $U \subset X$ open not containing the singularities of $X$, if $X$ has any.
		We define that $(X',[f']) \in \Met(S)$ is in $V(\epsilon, U)$ if and only if there exists a homeomorphism $\phi: X \to X'$ such that the restriction $\phi|_U$ is bilipschitz with constant less than $(1+\epsilon)$; and $\phi$ is compatible with the markings, i.e. $[\phi\circ f] = [f']$ as homotopy classes.
		The collection of all such neighbourhoods generates a topology on $\Met(S)$ that we call \emph{Lipschitz topology}.
	\end{definition}

	As usual, as a space of marked geometric structures, we can describe $\Met(S)$ as the quotient by the action of $\Diff_0^+(S)$ via pull-back on the space of metrics on $S$ coming from rescaling the Thurston metric of a projective structure or from a holomorphic quadratic differential. It is easy to see that the Lipschitz topology on $\Met(S)$ comes from the topology on the space of metrics on $S$ given by uniform convergence on compact sets in the complement of a finite set of singularities.
	
	\begin{remark}
		As a consequence, we have that the area functional $\Area: \Met(S) \to \R^+$ is continuous, because the integral of the volume form converges under uniform convergence of the metric on compact sets, as the metric is bounded near the singularities as they are of cone type, with finite bounded cone angle.
	\end{remark}
	
	\begin{proposition}  \thlabel{th:proper-action}
		The action of $\Mod(S)$ on $\Met(S)$ by pre-composition on the marking is continuous and properly discontinuous.
	\end{proposition}
	\begin{proof}
		Let us notice that from the definition, the Lipschitz topology on $\Met(S)$ is invariant under the action of $\Mod(S)$ that is then continuos.
		The proof of the proper discontinuity of the action of $\Mod(S)$ on Teichmüller space in \cite[Section 12.3.3]{farbmarg} adapts easily to our slightly broader context. The key observations to extend the argument are the following.
		The length of every closed geodesic in $\Gr(X,\lambda)$ is longer then the corresponding curve on $X$.
		In particular, the property that only finitely many curves have length bounded by any constant $L$ they show for hyperbolic surfaces, extends to grafted metrics, and then also their rescaled copies.
		The same property also holds for any half-translation surface as all the closed geodesics of length bounded by $L$ are concatenations of saddle connections of length at most $L$, which are in finite number.
		Moreover, since on $\Met(S)$ we have the Lipschitz topology, we have a similar control on how much length functions can change in a neighbourhood $V(\epsilon,U)$ of a metric.
	\end{proof}

	\begin{definition}
		We can then define the quotient space $\met(S) = \Met(S)/\Mod(S)$. It is the space of (unmarked) rescaled grafted surfaces  and half-translation surfaces, all considered up to isometry.
	\end{definition}

	\begin{proposition}
		The map $\PT(S) \to \Met(S)$ given by the Thurston metric is a $\Mod(S)$-equivariant proper embedding.
	\end{proposition}
	\begin{proof}
		We now show that this is an injective map, as it is possible to recover the grafting data from its Thurston metric.
		Indeed, the interior of the hyperbolic part is given by the set of points with a neighbourhood locally isometric to $\H^2$. The metric on such pieces and the relative position of such pieces is enough to determine the hyperbolic structure $X$ (see \cite{CF21}).
		On the other side, the complement of the hyperbolic part will be a geodesic lamination $\lambda$, possibly with foliated cylinders in place of isolated leaves (see Section \ref{sec:construction}).
		The transverse measure is uniquely determined by the area (in the sense of volume measure as in the remark in the next subsection)
		of each irreducible component of $\lambda$. Indeed, the transverse measure supported by each irreducible component is unique up to rescaling.

		The inclusion is $\Mod(S)$-equivariant and the action is proper on both space, by \thref{th:proper-action} and because $\PT(S) \cong \T(S) \times \ML_0(S)$ is $\Mod(S)$-equivariant and the action is proper on $\T(S)$.
		So in order to show that the immersion $\PT(S) \to \Met(S)$ is en embedding, it is enough to check properness of the induced map at the quotient.
		$ \PT(S)/\Mod(S) = \PM(S) \to \met(S)$.

		We observe that the subset of $\PM(S)$ described by pairs $(X,\lambda)$ where $X$ is $\epsilon$-thick and $\ell_X(\lambda)$ is bounded is compact, as it is a bundle over the $\epsilon$-thick part of moduli space, which is compact, and where every fibre is also compact. So any diverging sequence of grafting data in $\PM(S)$ has either $\ell_X(\lambda)$ going to infinity, or $\ell_X(\lambda)$ is bounded and the metric $X$ has systole going to 0.
		In both cases the metric of the grafted surface exits every compact set.
		Indeed, in the first case clearly the area of the surface $\Gr(X,\lambda)$ is going to infinity and in the second one the diameter is going to infinity, which both imply divergence in $\Met(S)$.
		So the inclusion $\PT(S) \to \Met(S)$ is a proper embedding. The $\Mod(S)$-equivariance follows from the definitions of the actions.
	\end{proof}

	Let us consider now flat surfaces coming from half-translation surface structures.
	A quadratic differential $q \in\QT(S)$ determines a half-translation surface, which is a marked surface with a flat singular metric, or in other words we have a map $\QT(S) \to \Met(S)$.
	However, this is not injective as the pair of metric surface with marking associated to $q$ is unique only up to the $S^1$ action by multiplying $q$ with unit complex numbers, corresponding to a rotation of the half-translation structure.
	Indeed, the metric is enough to determine the period coordinates of $q$, up to the rotation factor.

	\begin{proposition}
		The natural map $\QT(S)/S^1 \to \Met(S)$ is a $\Mod(S)$-equivariant closed embedding.
	\end{proposition}
	\begin{proof}
		The map $\QT(S)/S^1 \to \Met(S)$ is continuous as the singular flat metric depends continuously on $q$, and is also injective as we noticed in the previous remark.
		Its inverse is also continuous, as the period coordinates depend continuously on the metric with respect to the Lipschitz topology, so the map is an embedding.

		It is closed, because its image is closed.
		We can indeed observe that its complement, the space given by rescaled grafted metrics, is open in $\Met(S)$.
		Indeed, any sequence of half-translation surfaces cannot converge in $\Met(S)$ to a surface with no cone singularities, as one can observe that uniform convergence of the metric preserves cone angles.
	\end{proof}

	\begin{corollary}
		By mapping class group equivariance, we obtain that also the maps $\PM(S) \to \met(S)$ and $\QT(S)/S^1 \to \met(S)$
		are embeddings.
	\end{corollary}

	\subsection{Area normalization}

	We defined the space $\Met(S)$ so that it supports a scaling action by $\R^+$.
	So we can take its projectification as the quotient by this action $\P\Met(S) = \Met(S)/\R^+$.
	Since the area is continuous on $\Met(S)$, we can identify $\P\Met(S)$ with the subspace of marked surfaces with unit area. 
	More precisely, we have the following.

	\begin{proposition}	\thlabel{th:grafted-area}
		The area of a grafted surface $\Gr(X,\lambda)$ is $2\pi |\chi(S)| + \ell_X(\lambda) $.
	\end{proposition}

	\begin{proof}
		The area is continuous so it is enough to show that the statement holds for surfaces $\Gr(X,\mu)$ where $\mu = \sum_i a_i \gamma_i$ is a weighted multicurve.
		
		For this case it is immediate as the area of the hyperbolic part is the area of the hyperbolic surface $X$, which is $2\pi |\chi(S)|$ by Gauss-Bonnet theorem, while the flat part is made of a finite union of cylinders whose area is given exactly by the product of their height and circumference $a_i \ell_X(\gamma_i)$, which by linearity add up to $\ell_X(\mu)$.
	\end{proof}

	\begin{remark}
		While we can compute the area of the hyperbolic part as integral of the volume form, because it is open subset, the same does not hold for the grafted region. Indeed, in the generic case of a lamination with no closed leaves, the grafted region has empty interior, although the integral of the volume form on the hyperbolic part is strictly less than the integral on the whole grafted surface.
		A working notion of area is given instead by the volume measure induced by the volume form on the grafted surface. The mass of the grafted region with respect to such a measure is $\ell_X(\lambda)$, as it is the difference between the total area and the one of the hyperbolic part.
	\end{remark}

	Let us consider now each grafted surface $\Gr(X,\lambda)$ to be rescaled to have unit area,
	meaning that its Thurston metric is rescaled by a multiplicative factor $k^{-2}$ such that 
	$$ k^2 = 2\pi |\chi(S)| + \ell_X(\lambda) .$$
	This means that instead of a hyperbolic part, $\Gr(X,\lambda)$ has a part with constant negative curvature $-k^2$, of area 
	$$\frac{2\pi |\chi(S)|}{2\pi |\chi(S)| + \ell_X(\lambda)} .$$

	We notice that for $\ell_X(\lambda)$ going to infinity, and so $k^2$ too, the negatively curved part of $\Gr(X,\lambda)$  will have curvature going to $-\infty$ and area going to 0. Since the total area is fixed to be one by the renormalization, this means that, on the other side, the area of the flat grafted part, for $\ell_X(\lambda)$ going to infinity, will tend to 1.
	This suggests that the larger the grafting parameter $\ell_X(\lambda)$, the more the area-normalized grafted surface will look flat.

	\begin{proposition}
		The composition $\PT(S) \to \Met(S) \to \P\Met(S)$ is also an embedding. 
	\end{proposition}
	\begin{proof}

		Thinking of $\P\Met(S)$ as the subspace of unit-area surfaces, the map is given by rescaling the Thurston metric of $\Gr(X,\lambda)$ by $k^{-2}$ as above. The map is clearly continuous and injective.
		The inverse is also continuous, as by looking at an area-normalized grafted surface, the rescaling factor $k^{-2}$ is the inverse of the constant curvature of its negatively curved part, which also varies continuously with respect to the Lipschitz topology.
		The continuous dependency is not true for the curvature of a single point (that is not even granted is well-defined) but it is for the constant curvature of an open subset, as in our case.
	\end{proof}

		On the other side, by rescaling a quadratic differential $q$, the area of the associated half-translation surface is rescaled accordingly.
		In particular, the embedding $\QT(S)/S^1 \to \Met(S)$ if equivariant with respect to the rescaling action by $\R^+$, so it descends to an embedding at the quotient, meaning from $\QT(S)/S^1/\R^+ = \QT(S)/\C^* = \P_\C\QT(S)$ to $\P\Met(S)$.
		Such an embedding is also closed as the complement is open as it is the projection to $\P\Met(S)$ of $\PT(S) \subset \Met(S)$, which is open, as complement of the image $\QT(S)/S^1$ which we showed being closed.

	\begin{proposition}
		The map $\QT(S)/\C^* \to \P\Met(S)$ is a $\Mod(S)$-equivariant closed embedding. 
	\end{proposition}

	Our main result is then the following.

	\begin{theorem}		\thlabel{th:bordification}
		By considering $\PT(S)$ and $\P_\C\QT(S)$ as subspaces of $\P\Met(S)$, we have that $\PT(S)$ is dense in $\P\Met(S)$ and in particular that $\P_\C\QT(S)$ is its boundary according to the Lipschitz topology. We shall write then
		$$ 	\P\Met(S) = \bar{\PT(S)} = \PT(S) \cup \P_\C\QT(S) $$
	\end{theorem}

	This means that every marked unit-area half-translation surface is adherent to $\PT(S)$.
	In particular, there is a sequence of marked area-normalized grafted surfaces converging to it with respect to the Lipschitz topology.
	Moreover, by $\Mod(S)$-invariance the same holds at the quotient.
	
	\begin{corollary}
		By considering $\PM(S)$ and $\P_\C\QM(S)$ as subspaces of $\P\met(S)$, we have that $\PM(S)$ is dense and in particular that $\P_\C\QM(S)$ is its boundary according to the Lipschitz topology.
	\end{corollary}

\subsection{Gromov-Hausdorff distance}

	The space $\met(S)$ can be endowed with the Gromov-Hausdorff distance, defined in general for compact metric spaces.
	We report here its definition.
	
	\begin{definition}[Gromov-Hausdorff distance]
		Given two compact metric spaces $(X, d_X), (Y, d_Y)$ their Gromov-Hausdorff distance $d_{GH}(X,Y)$ is defined as
		$$ d_{GH}(X,Y) = \inf \{ d_H(X,Y) \,\mid\, d \text{ distance on } X\cup Y, \; d|_X = d_X,\; d|_Y = d_Y \} $$
		where $d_H$ is the Hausdorff distance induced by the distance $d$ on $X\cup Y$.
	\end{definition}

	A more practical approach to Gromov-Hausdorff distance, that we will use, is given by the following lemma, by using $\epsilon$-isometries. It follows for example from the discussion on $\epsilon$-dense subsets in \cite[Section 11.1.1]{petersen06}. A complete proof can also be found in the lecture notes \cite[Theorem 6.14]{tuzhilin20}.

	\begin{definition}
		An $\epsilon$-isometry is a map $f:X\to Y$ between metric spaces that is an isometry up to an additive error, or more precisely such that
		\begin{enumerate}[i.]
			\item for all $x_1, x_2 \in X$,
			$d_X(x_1, x_2) - \epsilon \leq d_Y(f(x_1), f(x_2)) \leq d_X(x_1, x_2) + \epsilon$
			\item for every $y\in Y$ there is $x\in X$ with $d_Y(f(x), y) \leq \epsilon$
		\end{enumerate}
	\end{definition}

	\begin{lemma}[Criterion for Gromov-Hausdorff convergence]		\thlabel{th:GH-criterion}
		If there is an $\epsilon$-isometry between the spaces $X,Y$, then $d_{GH}(X,Y) \leq 2\epsilon$. 
	\end{lemma}

	The Lipschitz topology is finer than the Gromov-Hausdorff and in general they do not coincide. However, in our case, thanks to the regularity of the metrics we are considering, under certain conditions we can upgrade the Gromov-Hausdorff convergence to Lipschitz convergence. The following is essentially a rewriting of Theorem 1.8 and Claim 4.3 in \cite{nagano02} in our more restricted setting.
	For this translation we just need to observe that their assumption on the lower bound of what they call CAT$_0$-radius is implied by a lower bound on injectivity radius and the surfaces being locally CAT(0).

	\begin{theorem}[\cite{nagano02}]		\thlabel{th:nagano}
		Let $(X_n)_n \subset \met(S)$ be a sequence of surfaces with uniformly lower bounded injectivity radius, and let $X\in\met(S)$ be
		such that $d_{GH}(X_n, X) = \epsilon_n$ with $\epsilon_n$ going to zero for $n$ going to infinity.
		Let $\{x_i\}_{i=1}^{k}$ be the set of cone singularities of $X$.
		Then for every $\epsilon> 0$, for $n$ large enough there is a $(1+\epsilon)$-bilipschitz homeomorphism $\phi_n: V_n \to U_n$ where $V_n \subset X_n$ and $U_n \subset X$ are such that
		\begin{itemize}
			\item $X\setminus U_n$ is a union of $k$ discs, each contained in a ball of radius $\epsilon$ centred in a singularity $x_i$; 
			\item $X_n \setminus V_n$ is the union of $k$ discs in $X_n$, each contained in a ball of radius $\epsilon$.
		\end{itemize}
	\end{theorem}

	We can then use this result to show that Gromov-Hausdorff convergence under certain additional conditions can be promoted to Lipschitz convergence. This is also done in a marking preserving fashion.

	\begin{lemma}		\thlabel{th:hausdorff-to-lipschitz}
		Consider the sequence of marked surfaces $(X_n, [f_n]) \in \Met(S)$ and $(X, [f])\in\Met(S)$. Let $\psi_n: X_n \to X$ be homotopy equivalences that are $\epsilon_n$-isometries with $\epsilon_n \to 0$ for $n\to \infty$ and compatible with the markings i.e. $[\psi_n \circ f_n] = [g]$. Then the sequence $(X_n, [f_n])$ converges to $(X, [f])\in\Met(S)$.
	\end{lemma}
	\begin{proof}
		Let us first show that the injectivity radius of $X_n$ is uniformly lower bounded.
		In a non-positively curved surface $Y$
		the injectivity radius is half of the systole $\sys(Y)$, that is the length of the shortest essential closed curve. And since $\psi_n$ is a homotopy equivalence, and an $\epsilon_n$-isometry between $X_n$ and $X$, then 
		$$ \sys(X_n) \geq \sys(X) - 2\epsilon_n $$
		Since $\epsilon_n$ goes to 0, for $n$ large enough we have a positive lower bound for the injectivity radius of $X_n$.

		By \thref{th:GH-criterion} the sequence of $\epsilon_n$-isometries implies that $d_{GH}(X_n, X) < 2\epsilon_n$.
		Then we can apply \thref{th:nagano} and obtain for every $n$ large enough a $(1+\delta_n)$-bilipschitz homeomorphism $\phi_n: V_n \to U_n$ with $\delta_n$ going to zero, and where $V_n \subset X_n$ and $U_n \subset X$ are such that
		\begin{itemize}
			\item $X\setminus U_n$ is a union of $k$ discs, each contained in a ball of radius $\delta_n$ centred in one of the singularity $x_i$; 
			\item $X_n \setminus V_n$ is the union of $k$ discs in $X_n$, each contained in a ball of radius $\delta_n$.
		\end{itemize}

		We can then patch the homeomorphisms $\phi_n$ with homeomorphisms between the discs in the complement of $U_n$ and $V_n$ and extend in this way $\phi_n$ to homeomorphisms from $X_n$ to $X$, implying then convergence in $\met(S)$.

		We now show that the $\phi_n: X \to X_n$ constructed are also compatible with the markings, and so we have convergence of the marked surfaces in $\Met(S)$.
		Consider the compositions $\psi_n \circ \phi_n^{-1}: X \to X$. Every non-singular point of $X$ has a neighbourhood $U$ that is eventually contained in $U_n$. Then 
		for $n$ large enough the restriction of $\psi_n \circ \phi_n^{-1}$ to $U$ is a composition of a $(1+\delta_n)$-bilipschitz and an $\epsilon_n$-isometric maps.
		By Arzelà-Ascoli there is a subsequence converging uniformly on compact sets of $X\setminus F$ to an isometry $h: X \to X$, where $F$ is the set of cone singularities of $X$. 

		Let us show that for $n$ big enough $\psi_n \circ \phi_n^{-1}$ is in the same homotopy class as $h$. 
		Fix a finite set of simple closed curves $\gamma_i$ in $X$ disjoint from $F'$, and generating $\pi_1(X)$.
		For $n$ large enough $U_n$ contains all such $\gamma_i$. Let us choose $\delta$ smaller than the injectivity radius of $X$. For uniform convergence of $\psi_n\circ \phi_n^{-1}$, for $n$ large enough $\psi_n \circ \phi_n^{-1}|_{U_n}$ is $\delta$-close to $h|_{U_n}$.
		Then it is easy to see that $(\psi_n \circ \phi_n^{-1})_* \gamma_i$ must be homotopic to $h_* \gamma_i$. But if $\psi_n \circ \phi_n^{-1}$ and $h$ induce the same map on the fundamental groups, then they are homotopic (as the mapping class group of the surface $S$ is maps invectively to the group of outer automorphisms of $\pi_1(S)$).
		And then $[\phi_n ]= [h^{-1} \circ  \psi_n]$, implying that 
		$[\phi_n \circ f_n] = [h^{-1} \circ  \psi_n \circ f_n] = [h^{-1} \circ g]$,
		and so
		$[ f_n] = [\phi^{-1} \circ (h^{-1} \circ g)]$.

		By running this same argument for any subsequence, we have that for every subsequence there exists $h$ a self isometry of $X$ such that, up to passing to a sub-subsequence, $(X_n, [f_n])$ converges in $\Met(S)$ to the same point $(X,[h^{-1}\circ g]) = (X,[g])$.
		Then the whole sequence converges to $(X,[g])$.
	\end{proof}

\section{Quantitative control on small grafting}		\label{sec:degrafting}

	We showed continuity of grafting, meaning in particular also that the grafted surface $\Gr(X,\lambda)$ converges in the sense of Gromov-Hausdorff to $\Gr(X,0) = X$ as $\lambda$ tends to the zero lamination. We show in this section a quantitative
	\footnote{Note that we didn't define a metric on the space of marked metrics $\Met(S)$, but only on the space of unmarked one $\met(S)$.
	Although the proof shows convergence in $\Met(S)$, the quantitative control on the convergence is actually done at the level of $\met(S)$, as the Gromov-Hausdorff distance is defined there.}
	control on the convergence, in particular we show that it is uniform in the $\epsilon$-thick part of Teichmüller space.

	\begin{proposition}		\thlabel{th:degrafting}
		If $X\in \T_\epsilon(S)$ then there is a constant $C_\epsilon$ depending only on $\epsilon$ such that
		$$ d_{GH} (\Gr(X,\lambda), X) \leq C_\epsilon \, \ell_X(\lambda). $$
	\end{proposition}


	\begin{proof}
		Let $\kappa: \Gr(X,\lambda) \to X$ be the collapsing map introduced by Thurston. We will show that $\kappa$ distorts the metric by at most an additive constant proportional to $\ell_X(\lambda)$. Then \thref{th:GH-criterion} will apply and conclude the proof.
		We know by \thref{th:kappa-lip} that $\kappa$ is 1-Lipschitz, so we are left to show that the map $\kappa$ does not shrink too much the distance between any two points.

		Let $x,y$ be points in $\Gr(X,\lambda)$. Let's consider their images $\kappa(x), \kappa(y)$ in $X$ and a geodesic arc $\alpha$ joining them, realizing their distance on $X$. We remind that with a slight abuse of notation we denote with $\alpha$ both the parametrization and the image of the arc as subset of the surface.
		Our goal is to pick an arc $\beta$ joining $x,y$, such that $\kappa(\beta) = \alpha$ and 
		\begin{equation}		\label{eq:beta-bound}
			\Len(\beta) \leq \Len(\alpha) + C_\epsilon \ell_X(\lambda)
		\end{equation}
		for a suitable constant $C_\epsilon$. Indeed, this will suffice to conclude that 
		$$ d(\kappa(x),\kappa(y)) = \Len(\alpha) \geq \Len(\beta) - C_\epsilon \ell_X(\lambda) \geq d(x,y) - C_\epsilon \ell_X(\lambda) .$$

		We have to distinguish here two cases: either $\alpha$ is contained in a leaf of $\lambda$, or it is transverse to $\lambda$, possibly with endpoints on $\lambda$. 

		\emph{First case.}
		If $\alpha$ is on a leaf of $\lambda$, then either the leaf has no atomic measure, in which case $\kappa$ restricts to an isometry between $\kappa^{-1}(\alpha)$ and $\alpha$, and we can just choose $\beta = \kappa^{-1}(\alpha)$; or $\alpha$ lies on a closed leaf $\gamma$ of $\lambda$ that is the image under $\kappa$ of a flat cylinder in $\Gr(X,\lambda)$. We choose in this case $\beta$ in $\kappa^{-1}(\alpha)$ joining $x,y$ such that it is concatenation of a geodesic $\beta_1$ parallel to $\gamma$ and $\beta_2$ orthogonal to $\beta_1$. It is easy to observe that $\Len(\beta_1) = \Len(\alpha)$ while $\Len(\beta_2)$ is less than the height $h$ of the cylinder.
		By additivity of the length function we have $h \cdot \ell_X(\gamma) \leq \ell_X(\lambda)$, and by assumption on the thickness of $X$, $\ell_X(\gamma) \geq \epsilon$, so 
		$$\Len(\beta_2) \leq h \leq \epsilon^{-1} \ell_X(\lambda).$$
		This implies \ref{eq:beta-bound} with $C_\epsilon = \epsilon^{-1}$.

		\emph{Second case.}
		We will now show that if $\alpha$ is transverse to $\lambda$, first that then $\kappa^{-1}(\alpha)$ is an arc, so we can choose $\beta = \kappa^{-1}(\alpha)$, and secondly that 
		\begin{equation}	\label{eq:len-extension}
			\Len(\beta) = \Len(\alpha) + i(\alpha,\lambda).
		\end{equation}
		This, together with \thref{th:intersection-bound} will imply \ref{eq:beta-bound} as we wanted.
		Note that actually if an endpoint of $\alpha$, say $\kappa(x)$, is on a closed leaf of $\lambda$, then $\kappa^{-1}(\alpha)$ might contain $x$ in its interior instead of having it as an endpoint. But even in this case $d(x,y) \leq \Len(\beta)$, so the argument still holds.

		The preimage $\kappa^{-1}(\alpha)$ of the arc $\alpha$ is still an arc in $\Gr(X,\lambda)$. Indeed, $\kappa$ is a local homeomorphism everywhere in the domain except on flat cylinders, each coming from the grafting along an isolated leaf $\gamma$ of $\lambda$ with positive mass, which is hence a closed geodesic.
		An intersection point $\alpha(t_0)$ of $\alpha$ with such a $\gamma$ has as preimage a straight segment in the cylinder, orthogonal to the boundary of it, that extends and joins the arcs $\kappa^{-1}(\alpha({t<t_0}))$ and $\kappa^{-1}(\alpha({t>t_0}))$. Note that this happens only on closed geodesics, which are finitely many and compact, so the total number of this kind of intersections is also finite. 

		We now want to show that \ref{eq:len-extension} holds true.
		We know that $\kappa$ is a local isometry on the hyperbolic pieces of $\Gr(x,\lambda)$, so the arc $\kappa^{-1}(\alpha)$ is basically $\alpha$ with some additional pieces inserted along the grafting locus. It is clear to see that in the case when $\lambda$ is a weighted multicurve $\lambda = \sum_i a_i\gamma_i$, the inserted grafted parts have lengths  exactly the sum of the heights of the cylinders crossed by $\alpha$ (with multiplicity), that is $\sum_i a_i i(\alpha,\gamma_i) = i\left(\alpha,\sum a_i \gamma_i\right) = i(\alpha, \lambda)$. So equation \ref{eq:len-extension} holds for weighted multicurves. The case for $\lambda$ general measured lamination follows from the continuity of grafting thanks to \thref{th:grafting-continuity}, since weighted multicurves are dense in $\ML(S)$ and one can also argue (with an Arzelà-Ascoli argument) that once fixed two points $x,y$ the arc $\alpha$ minimizing their length changes continuously upon the choice of the hyperbolic metric and of the measured lamination.

	\end{proof}
	
	We separated the statement of the following lemma from the previous proof, in that it is independent of grafting, and of independent interest.
	
	\begin{lemma}		\thlabel{th:intersection-bound}
		Let $X\in \T_\epsilon(S)$ be an $\epsilon$-thick hyperbolic surface and $\lambda\in\ML(S)$ a measured lamination. 
		If $\alpha$ is a geodesic arc in $X$ transverse to $\lambda$, and it is distance minimizing between its endpoints, then 
		$$ i(\alpha,\lambda) \leq C_\epsilon \cdot \ell_X(\lambda) $$
		where the constant $C_\epsilon$ depends only on $\epsilon$.
	\end{lemma}
	\begin{proof}
		Let us consider for every point $\alpha(t)$ in $\alpha \cap \lambda$, the geodesic segment $U_t$, which is an open neighbourhood of radius $\epsilon/4$ of $\alpha(t)$ along the leaf of $\lambda$ containing $\alpha(t)$. Note that it is actually just a segment, i.e. it does not form a closed curve since it would have length $\epsilon/2$, but we know by assumption that $X$ is $\epsilon$-thick.
		Moreover, all such segments $U_t$ are disjoint. 
		Suppose by contradiction that $U_t$ intersects $U_{t'}$, with wlog $t<t'$. Then they lie on the same leaf of $\lambda$ and so $\alpha(t)$ and $\alpha(t')$ will lie on the same leaf of $\lambda$ at distance strictly less than $\epsilon/2$. Since the injectivity radius in an $\epsilon$-thick surface is $\epsilon/2$, the segment on $\lambda$ joining $\alpha(t)$ and $\alpha(t')$ is length minimizing. But being $\alpha$ length minimizing, this implies that that segment lies in $\alpha$, which contradicts the hypothesis $\alpha$ transverse to $\lambda$.

		Consider now $U$ the subset of $\lambda$ formed by all the (disjoint) segments $U_t$ we considered
		$$ U := \bigcup_{t \in \alpha^{-1}(\lambda)} U_t .$$
		We remind that the length $\ell_X(\lambda)$ for a measured lamination $\lambda$ can be seen as the total mass of the measure supported on $\lambda$ given locally by the product $d\lambda \times d\ell$ where $d\lambda$ is the transverse measure associated to $\lambda$ and $d\ell$ is the Lebesgue measure along the leaves induced the hyperbolic metric $X$.
		Now since every segment $U_t$ has by construction length $\epsilon$, one can compute

		\begin{align*}
			\int_U d\lambda \times d\ell 
			&= \int_U d\lambda(t) \times d\ell_{U_t} \\
			&= \int_{\alpha} d\lambda(t) \cdot \epsilon/2\\
			&= i(\alpha,\lambda) \cdot \epsilon/2
		\end{align*}
		And since trivially $U \subset \lambda$, we also have
		\begin{equation*}
			i(\alpha,\lambda) \cdot \epsilon/2 = \int_U d\lambda \times d\ell \leq \int_\lambda d\lambda \times d\ell = \ell_X(\lambda)
		\end{equation*}
		That will complete the proof with $C_\epsilon = 2\epsilon^{-1}$.

	\end{proof}

	We chose to give a proof of the previous lemma that does not make use of grafting, since its statement does not mention grafting as well. Although, the proof above has a very geometric interpretation involving grafting. 
	
	\begin{remark}
		From the proof above, consider the preimage $U'$ of $U$ through the collapsing map $\kappa:\Gr(X,\lambda) \to X$. Then $U'$ is a subset of the grafted region. In particular, it is the subset of the grafted region, made of points that lies at distance $\epsilon/4$ from the arc $\kappa^{-1}\alpha$.	We remind that the transverse measure $\lambda(\alpha)$ gives exactly the total grafted height along all the leaves crossed by $\alpha$. So one can easily compute the area of $U$ as the product $\lambda(\alpha) \cdot \epsilon/2$. And by containment this must be less than the total grafted area, which is $\ell_X(\lambda)$ (see \thref{th:grafted-area}).
		Leading to the exact same computation as in the proof above.
	\end{remark}
	
	We invite the reader to remember this interpretation as it will be the main idea to prove the complementary result \thref{th:inflated-part-bound} for the deflation of grafted surfaces.

\section{Orthogeodesic foliation} \label{sec:orthogeodesic}

In this section we proceed introducing the orthogeodesic foliation: a measured foliation of the entire surface, transverse to a given measured lamination and whose construction depends on the hyperbolic metric of the surface. It was already used in \cite{CDR10}, but has been studied more deeply by Calderon and Farre in \cite{CF21} and \cite{CF24}. We will report some results about it, and then in the last subsection we perform our own construction of a map, called deflation, that we will use in the next section to measure how close are grafted surfaces to half-translation ones.

	\subsection{Construction of the orthogeodesic foliation}

	\begin{definition}
		Given $X\in\T(S)$ a hyperbolic surface and $\lambda$ a measured lamination on it, the \emph{orthogeodesic foliation} $\O_\lambda(X)$ is the piecewise $C^1$ singular foliation where the leaves are the fibres of the closest point projection $p: X \to \lambda$.
		Moreover, $\O_\lambda(X)$ is endowed with a transverse measure induced by the hyperbolic length along the leaves of $\lambda$.
	\end{definition}

	Let us realize first the same construction on the universal cover.
	Let us consider the lift $\tilde \lambda$ of $\lambda$ in $\H^2$. It is a collection of disjoint geodesics lines. We observe that the nearest point projection $p$ is not well-defined everywhere, since there are points equidistant to two or more leaves of $\tilde \lambda$. It is easy to observe that all such points form a geodesic spine, i.e. a graph where edges are also allowed to have only one endpoint.
	More precisely one can easily see that such a spine $\tilde \Sp$ has as vertices the points equidistant to at least three leaves of $\tilde \lambda$ and its edges are geodesic segments between the vertices or infinite geodesic rays emanating from the vertices.
	The whole construction is invariant under the covering automorphisms, so passing to the quotient it gives a spine $\Sp \subset X$.
	
	\begin{example}
		When the support of $\lambda$ is a multicurve, $\Sp$ is a (finite) graph. Infinite edges are possible only when there is at least one region in the complement of $\lambda$ with infinite diameter.
	\end{example}
	
	\begin{example}
		When $\lambda$ is maximal, so its complementary regions are ideal triangles, $\Sp$ is a disjoint union of tripods: for every ideal triangle we have a vertex of $\Sp$ in the centre and one geodesic ray for each ideal vertex.
	\end{example}
	
	In the complement of $\tilde\Sp$ in $\H^2$ the closest point projection to $\lambda$ is well-defined and smooth, then its fibres determine a smooth foliation of $\H^2 \setminus \tilde \Sp$. Such foliation, depending geometrically only on $\tilde\lambda$, projects down to a foliation of $X \setminus \Sp$.
	There are a few verifications needed to check that this foliation matches nicely to give a foliation on the entire surface $X$ and that the transverse measure is well-defined. The details are discussed in \cite{CF21}. We summarize the result in the following proposition.

	\begin{proposition}[{\cite[Section 5.2]{CF21}}]
		The foliation described on $X \setminus \Sp$ extends (after a small homotopy of $\H^2$ supported on a small neighbourhood of $\Sp$) to a well-defined smooth measured foliation $\O_\lambda(X) \in \MF(\lambda)$ on $X$.
		Moreover, $\O_{\lambda}(X)$ meets $\lambda$ orthogonally and the transverse measure it induces along the leaves of $\lambda$ coincides with the hyperbolic length.
	\end{proposition}
	
	\begin{remark}
		Notice that by the construction the orthogeodesic foliation $\O_\lambda(X)$ intersects $\lambda$ orthogonally, but this does not imply a priori that the pair of foliations $(\O_\lambda(X), \lambda)$ can be realized transversely, meaning that $O_\lambda(X) \in \MF(\lambda)$. However, this is checked in \cite[Lemma 5.8]{CF21}.
	\end{remark}

	\subsection{Main properties of the orthogeodesic foliation}

		The orthogeodesic foliation determines completely the hyperbolic metric of the surface, indeed the following holds.
		\begin{theorem}[{\cite[Theorem D]{CF21}}]		\thlabel{th:orthogeodesic}
			The map $O_\lambda: \T(S) \to \MF(\lambda)$ is a $\Mod(S)$-equivariant homeomorphism.
		\end{theorem}
		And in general by combining all the maps $\O_\lambda$ for all $\lambda \in \ML(S)$, one gets the following mapping
		$$ \O: \T(S) \times \ML(S) \to \QT(S) ; \quad \O(X,\lambda) = q(\O_\lambda(X), \lambda) .$$
		
		\begin{remark}
			Since $\QT(S)$ can be seen as the subset of  $\MF(S) \times \MF(S)$ made of pairs of transverse foliations, thanks to the previous theorem and the natural correspondence between measured laminations and measured foliations, it follows immediately that $\O$ is a bijection.
			Moreover, due to its construction of geometric nature, $\O$ is mapping class group equivariant, meaning that for any $\phi\in \Mod(S)$, one has $\O(\phi_*X, \phi_* \lambda) = q( \phi_*\O_\lambda(X),  \phi_*\lambda) = \phi_*q( \O_\lambda(X),  \lambda) $.
		\end{remark}

		We point out that the mapping is not a morphism of bundles, as the metric on the half-translation surface $q(\O_\lambda(X), \lambda)$ will not be conformal to $X$. The mapping $\O$ is not even continuous.
		Although not globally continuous, the same authors on a more recent work, showed that $\O$ has many large domains on which it is continuous, with continuos inverse.
		
		\begin{theorem}[{\cite[Theorem A]{CF24}}]	\thlabel{th:O-continuity}	
			Suppose that $\lambda_n \to \lambda$ in the measure topology, and the supports of $\lambda_n$ also converge to the one of $\lambda$ in the Hausdorff topology. Then $(X_n, \lambda_n) \to (X, \lambda)$ in $\T(S)\times\ML(S)$ if and only if $\O(X_n, \lambda_n) \to \O(X, \lambda)$ in $\QT(S)$.
		\end{theorem}

	\subsection{Crowned surfaces}		\label{sec:crowned-surfs}

		A hyperbolic surface with crowned boundary is a complete, finite-area hyperbolic surface with totally geodesic boundary, with punctures on the boundary such that each boundary component is either a closed geodesic, if it does not have punctures, or a hyperbolic crown. This means that the segments between the punctures are infinite geodesics, and the two segments adjacent to a puncture are asymptotic and the puncture correspond to their common point at infinity.

		Hyperbolic surfaces with crowned boundary, or shortly, crowned surfaces arise as the connected components of the complement of a geodesic lamination on a hyperbolic surface.
		\thref{th:orthogeodesic} tells us that after having fixed $\lambda$, the vertical foliation of $\O(X,\lambda)$ determines the hyperbolic structure $X$, so in particular determines the metric on every component of $S\setminus \lambda$. A better description of the hyperbolic structure of each of those crowned surface is showed in \cite[Section 6]{CF21}.
		Indeed, the proof of \thref{th:orthogeodesic} consists in showing first that the orthogeodesic foliation determines the metric on each component of $X\setminus \lambda$, and then that it encodes also the shear between those, which means, roughly speaking, how the different components are put together.
		
		If we look at a component $A$ of $X\setminus \lambda$, the information of the hyperbolic metric kept by the orthogeodesic foliation  $A$ can be condensed to just the embedding of the spine $\Sp$ in $A$ and the transverse measure induced by the measured foliation $\O_\lambda(X)$ on each of the edges of $\Sp$.
		These data are enough to reconstruct the hyperbolic metric on $A$. See for example \cite[Section 6]{CF21}, 
		\cite{mondello09}, \cite{AHC22}.

		\begin{remark}
			Moreover, we also notice that $A$ deformation retracts to $\Sp$, that is, A is a ribbon graph modelled on $\Sp(S)$.
			Hence, there is a correspondence between crowned surfaces with a hyperbolic metric and their metric ribbon graphs $\Sp(A)$.
		\end{remark}

		This is stated more precisely in \cite{AHC22}, but only for compact surfaces with geodesic boundary (non-crowned). An equivalent discussion in the general case of crowned surfaces is carried in \cite[Section 6]{CF21}, but there instead of metric ribbon graph, a dual datum, given by arc systems, is used. The correspondence between the two points of view with metric ribbon graphs and weighted arc systems is discussed in \cite{mondello09}.

	\subsection{Hausdorff approximation of measured laminations}

	The Hausdorff topology on the space of geodesic laminations is given by fixing a hyperbolic structure $X$ for $S$ and taking then the induced Hausdorff distance between geodesic laminations as compact subsets of $X$.
	We observe that if we have a sequence of measured laminations $\lambda_n$ converging as measures to another measured lamination $\lambda$, this does not necessarily imply the Hausdorff convergence of their support. The Hausdorff limit, for $n$ going to infinity, of the supports of $\lambda_n$ will contain the one of $\lambda$, but the containment can be strict, as there can be leaves of $\lambda_n$ whose mass converges to 0.

	Nevertheless, even requiring convergence of the supports, any measured lamination can still be approximated by weighted multicurves.
	In \cite[Section 4.2]{notesonnotes} it is shown that every measured lamination is uniquely decomposed as a finite union of irreducible components and each irreducible component is approximable in both the measure and in the Hausdorff sense by a single weighted closed geodesics. This implies the following.
	
	\begin{lemma}		\thlabel{th:hausdorff-approx}
		For every $\lambda\in\ML(S)$ there exists a sequence of weighted multicurves $(\mu_n)_n$ that converges to $\lambda$ in the measure topology and whose supports converge to the support of $\lambda$ with respect to the Hausdorff topology.
	\end{lemma}

	A more precise characterization of the combined measure and support convergence of measured lamination is given in \cite[Lemma 15.1]{CF24}.

	\subsection{Deflation}

	The main construction we will use to prove our main result is what we call deflation map: a map between the grafted surface $\Gr(X,\lambda)$ and the half-translation surface $\O(X,\lambda)$ that allows us to relate their geometries. In \cite{CF21} they describe already a version of this map and show a property of it.
	
	\begin{proposition}[{\cite[Proposition 5.10]{CF21}}]
		Given a marked hyperbolic structure $(X,[f]) \in \T(S)$ and $\lambda \in \ML(S)$, let $(Z,[g]) = \pi(q(\O_{\lambda}(X), \lambda)) \in \T(S)$ be the marked complex structure on which $q(\mathcal{O}_{\lambda}(X), \lambda)$ is holomorphic. There is a map
		$$ D: X \to Z $$
		that is a homotopy equivalence restricting to an isometry between $\Sp$ with its metric induced by integrating the edges against $\O_{\lambda}(X)$ and the graph of horizontal saddle connections of $q(\O_{\lambda}(X), \lambda)$ with the induced path metric. Moreover, $D \circ f \simeq g$ and $D_* \O_{\lambda}(X) = \re(q)$ and $D_* \lambda$ is equivalent to $\im(q)$ as measured foliations.
	\end{proposition}

	They show that $D_* \lambda$ is equivalent as measured foliation to $\im(q)$. We remind that for example if $\lambda$ is multicurve, it has a finite number of leaves, while $\im(q)$ has uncountably many, so $D_*(\lambda)$ will only be equivalent to $\im(q)$ via the laminations - foliation correspondence.
	
	What we will do is to replicate the construction, but replacing $X$ with the grafted surface $\Gr(X,\lambda)$ and $\lambda$ with $\hat \lambda$, a specific lamination on $\Gr(X,\lambda)$ that is equivalent to $\lambda$.
	In this case it is possible to define a different map $\D: \Gr(X,\lambda) \to \O(X,\lambda)$, that we will still call deflation, that actually sends $\hat\lambda$ to $\im(q)$, leaf by leaf.
	More precisely, we will show the following.
	
	\begin{proposition} \thlabel{th:deflation}
		Given $(X,\lambda) \in \T(S) \times \ML_0(S)$, let $q = \O(X,\lambda)$ be the associated half-translation surface.
		There is a homotopy equivalence 
		$$ \D: \Gr(X,\lambda) \to \O(X,\lambda) $$
		compatible with the markings of $\Gr(X,\lambda)$ and of $\O(X,\lambda)$ and	such that $\D_* \O_{\lambda}(X) = \re(q)$ and $\D_* \hat\lambda = \im(q)$.
	\end{proposition}

	\subsection{Construction of the deflation}
	\label{sec:construction}

		The idea of the construction of $\O(X,\lambda)$ and the deflation map $\D$ is in essence the same as in \cite{CB88}, where given a pair of transverse geodesic measured laminations on a hyperbolic surface, the half-translation surface associated to the pair is constructed.
		Here the starting surface will be $\Gr(X,\lambda)$ a grafted surface and the pair we consider is made of a measured foliation and a geodesic measured lamination, corresponding respectively to $O_\lambda(X)$ and $\lambda$.

		\paragraph*{The orthogeodesic foliation on the grafted surface}		
		We constructed $\O_\lambda(X)$ as a foliation on $X$. This induces naturally an equivalent foliation on $\Gr(X,\lambda)$ just by taking the preimage of $\lambda$ via Thurston's collapsing map $\kappa: \Gr(X,\lambda) \to X$ with the pull-back measure. Indeed, since $O_\lambda(X)$ intersects $\lambda$ orthogonally in $X$, one can see that its preimage under $\kappa$ is still a smooth foliation, equivalent to $O_\lambda(X)$.

		\paragraph*{The measured lamination on the grafted surface}
		The collapsing map $\kappa$ allows us to pull-back $\lambda$ to a measured lamination $\kappa^*\lambda$ on $\Gr(X,\lambda)$.
		We will now define $\hat\lambda$ a geodesic representative for $\kappa^*\lambda$, as, due to the piecewise flat metric, laminations do not have necessarily a unique representative geodesic with respect to the grafted metric.

		For example if $\lambda = a\gamma$ a simple closed geodesic on $X$ with mass $a>0$, then $\Gr(X,\lambda)$ contains a flat cylinder of length $a$ corresponding to the curve $\gamma$. We then choose $\hat\lambda$ to be the lamination foliating the cylinder by parallel copies of $\gamma$ with a uniform transverse measure, that is the one given by the Euclidean metric of the cylinder along the direction orthogonal to the foliation. This naturally extends to any multicurve $\mu=\sum_i a_i \gamma_i$, by putting $\hat\mu = \sum_i \widehat{a_i\gamma_i} $.
		
		In the general case $\lambda$ can always be decomposed as $\lambda = \mu + \lambda'$ with $\mu$ a weighted multicurve and $\lambda'$ a measured lamination with no closed leaves (see the decomposition into irreducible components in \cite[Section 4.2]{notesonnotes}). Then we can just define $\hat\lambda = \hat\mu + \kappa^*\lambda'$, where $\kappa^*\lambda'$ is simply the preimage $k^{-1}(\lambda')$ with the pull-back measure, as $\kappa$ is injective on $\kappa^{-1}(\lambda')$.
		For simplicity of notation, from now on, we will also call $\lambda$ the above defined measured lamination $\hat\lambda$ when it's clear that we are referring to the lamination on the grafted surface.

		\paragraph*{Multicurve case}
		We now have on the grafted surface $\Gr(X,\lambda)$ a lamination $\lambda$ and a singular foliation $\O_\lambda(X)$ transverse to it. We want to construct $\O(X,\lambda) = q(\O_\lambda(X), \lambda)$ the half translation surface with the pair of vertical and horizontal foliation  equivalent to $\O_\lambda(X)$ and $\lambda$.
		Let us consider now the easier case where $\lambda$ is a weighted multicurve on $X$.
		
		In this case $\Gr(X,\lambda)$ is just a gluing of finitely many cylinders and compact hyperbolic surfaces with geodesic boundary. The lamination $\lambda$ and the foliation $\O_\lambda(X)$, restricted to each cylinder, give a pair of transverse foliations, geodesic with respect to the Euclidean metric of the cylinder, the first one parallel to the boundary and the second one orthogonal to it. Each hyperbolic piece is instead only foliated by $\O_\lambda(X)$ and does not contain any leaf of $\lambda$.
		What we will do is completely collapsing each hyperbolic piece along the orthogeodesic foliation. More precisely this means taking the quotient with the following equivalence relation. Two points in the same hyperbolic piece are equivalent if they lie in the same (possibly singular) leaf of $\O_\lambda(X)$ restricted to the hyperbolic piece. 
		
		\begin{proposition}
			The quotient above yields a half-translation surface obtained by gluing the cylinders along maps between their boundaries which are piecewise isometries.
		\end{proposition}
		\begin{proof}
			Let us focus on a single hyperbolic piece $A$, and let us call for easiness of notation $\eta$ the orthogeodesic foliation restricted to $A$. The foliation $\eta$ has a finite number of singularities, from which finitely many singular leaves emanate. By the definition of orthogeodesic foliation, each of those singular leaves will be a geodesic segment between a singularity and the boundary of $A$, and orthogonal to it.
			Let us call $\Gamma$ the union of all the singular leaves. The boundary $\de A$ is cut by the endpoints of $\Gamma$ in finitely many segments $a_i$.	
			
			\begin{figure}[ht]
				\includesvg[width=0.5\textwidth]{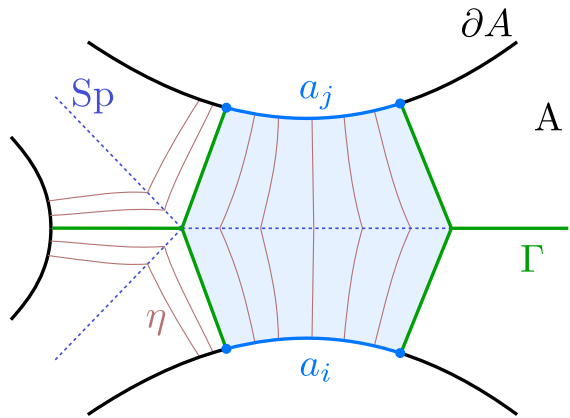}
				\centering
				\vspace{0.2cm}
				\begin{minipage}{0.7\textwidth}
					\caption{A band foliated by $\eta$, encoding the gluing between $a_i$ and $a_j$.}
					\label{fig:band}
				\end{minipage}
			\end{figure}
			
			We now observe that each connected component of $A \setminus \Gamma$ is foliated by a family of regular leaves of $\eta$, that, by definition of orthogeodesic foliation, are arcs with the two endpoints on $\de A$.
			This implies that each of those connected components is, like in Figure \ref{fig:band} a band connecting two segments $a_i, a_j$ on $\de A$. We also observe that it is made of two isometric pieces separated by the spine: the reflection along the spine induces a symmetry of the band, as by definition the segment of spine in the band is the locus of points equidistant from $a_i$ and $a_j$.
			In particular, the two segments $a_i, a_j$ have the same hyperbolic length. Indeed, by definition, the transverse measure of the orthogeodesic foliation $\eta$ induces the hyperbolic length along $\de A$, so the pairing between $a_i$ and $a_j$ given by the leaves of $\eta$ gives an isometry between the two segments.
			
			Let us now consider singular leaves. For every singularity, collapsing the $k$-pronged star of singular leaves through it, is equivalent to removing it and gluing its $k$ endpoints on $\de A$ together. But this is equivalent to considering, in the gluing described above, the segments $a_i$ to be closed, that is with endpoint included.
			
			In conclusion the quotient of $\Gr(X,\lambda)$ obtained by collapsing the leaves of $\eta$ is equivalent to removing the interior of the hyperbolic piece $A$ and then gluing its boundary (which lies on the boundary of the cylinders adjacent to the hyperbolic piece $A$) along isometries between the segments $(a_i)_i$ as we described.
			
			By repeating this construction for each of the finitely many hyperbolic pieces, we are left with just a collection of cylinders with their boundary components glued by piecewise isometries. The resulting surface $q$ is a half-translation surface.
			Indeed, every cylinder can be realized as a quotient of a rectangle in $\R^2$ with vertical sides identified. Then the gluing above is realized by piecewise isometries between their horizontal sides.
		\end{proof}

		The deflation map $\D:\Gr(X,\lambda) \to \O(X,\lambda)$ in this case is just the map collapsing the hyperbolic pieces of the grafted surface along the orthogeodesic foliation, essentially collapsing them onto their spine. 
		By construction of $\O_\lambda(X)$ and $\lambda$ on the grafted surface, $\D$ sends them to the pair of vertical and horizontal foliations on $\O(X,\lambda)$. And its restriction to the spine $\D|_{\Sp}$ is a homeomorphism to the gluing locus between the cylinders, that is to the diagram of horizontal separatrices of $\O(X,\lambda)$.

		\paragraph*{Geometric train track and rectangle decomposition}
		For the general case we will essentially follow along the lines the proof of \cite[Proposition 5.10]{CF21}, adding more details in order to better understand the additional metric properties of our version of the deflation map $\D$. The construction is based on Thurston's geometric train track construction (also explained in \cite[Construction 5.6]{CF21}) and uses a rectangle decomposition introduced in \cite{CDR10}. In this last work, like in this article, the starting surface is the grafted one, but the map produced is different as what they seek to obtain is a quasi-conformal map, while we care instead about other geometric properties.

		Let us consider $\delta>0$ such that the metric $\delta$-neighbourhood $N_\delta(\lambda)$ of $\lambda$ in $\Gr(X,\lambda)$ does not contain any of the finitely many singular points for the orthogeodesic foliation $\O_\lambda(X)$.
		The orthogeodesic singular foliation restricts to a regular foliation $\eta$ on the subset $N_\delta(\lambda)$ of $\Gr(X,\lambda)$. The boundary of $N_\delta(\lambda)$ is not geodesic, but by definition it is a level set of the distance function from a leaf of $\lambda$, and thus orthogonal to the leaves of the orthogeodesic foliation as they realize the paths of minimum distance to $\lambda$.
		Hence, $N_\delta(\lambda)$ is also called \emph{geometric train track} for $\lambda$ and the restriction of the leaves of $\O_\lambda(X)$ to it are called \emph{ties} of the train track.
		The geometric train track $N_\delta(\lambda)$ deformation retracts to a train track $\tau \subset N_\delta(\lambda)$ by collapsing each tie. We also observe that by construction $\tau$ snugly carries $\lambda$, meaning that the transverse measure induced by $\lambda$ of each branch of $\tau$ is strictly positive.

		\begin{figure}[ht]
			\includesvg[width=0.45\textwidth]{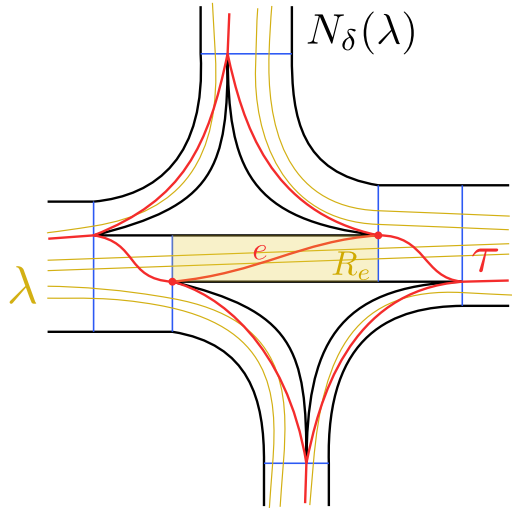}
			\centering
			\vspace{0.2cm}
			\begin{minipage}{0.8\textwidth}
				\caption{Example of geometric train track $N_\delta(\lambda)$ with its decomposition in rectangles $R(e)$'s.}
				\label{fig:geom-train-track}
			\end{minipage}
		\end{figure}
		We remind that $\tau$ has finite combinatorics, meaning a finite amount of switches and branches.
		To each branch $e$ of $\tau$ corresponds a region $R(e)$ of $N_\delta(\lambda)$ which we may think as a rectangle (although only the sides along the ties are geodesics), which is the union of ties crossing $e$. In this way we divided $N_\delta(\lambda)$ into rectangles, which are pairwise either disjoint or adjacent along a tie through a switch of $\tau$.
		
		We partitioned the grafted surface $\Gr(X,\lambda)$ in finitely many rectangles inside $N_\delta(\lambda)$ and finitely many connected complementary regions of $N_\delta(\lambda)$. We will now explicitly construct the half-translation surface $\O(X,\lambda)$ and the deflation map $\D: \Gr(X,\lambda) \to \O(X,\lambda)$.
		
		\paragraph{The deflation map locally}
		Let us first consider one rectangle $R(e)$. As it is a union of ties, we can define its height $h$ as the total transverse measure induced by the orthogeodesic foliation, and its length $l$ as the total mass given by the transverse measure of $\lambda$ to any of the ties in the interior of $R(e)$. The length is well-defined as all such ties are transverse to $\lambda$ and transversely homotopic one to the other.
		
		Let us now consider a Euclidean rectangle $Q(e)$ with same length $l$ and height $h$. We can map $R(e)$ to it, such that the transverse measures with respect to $\lambda$ and $\O_\lambda(X)$ are mapped to the vertical and horizontal coordinates in the rectangle.
		More precisely we define the map 
		$$f_e: R(e) \to Q(e) = [0,l]\times [0,h]$$
		as follows.
		We choose one of the two sides of $R(e)$ along the ties, call it $a$. The orientation of the surface together with the choice of $a$ induces an orientation for the ties. Then we define $f_e$ as in Figure \ref{fig:deflation-between-rectangles}.

		\begin{figure}[ht]
			\includesvg[width=0.8\textwidth]{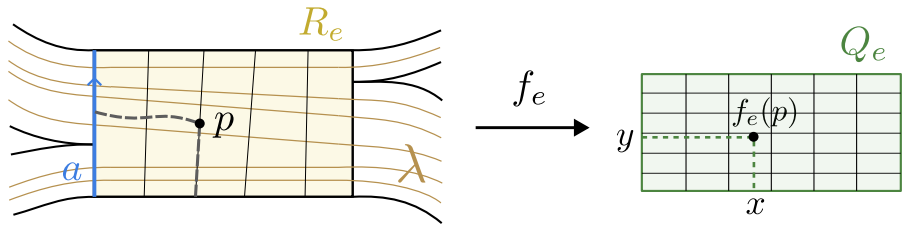}
			\centering
			\vspace{0.2cm}
			\begin{minipage}{0.8\textwidth}
				\caption{
				Any point $p$ in $R(e)$ is mapped to the point $(x,y)$ in $Q(e)$, where $y$ is the transverse measure with respect to the orthogeodesic foliation given to an arc joining $p$ to $a$ and transverse to the ties; and $x$ is the transverse measure given by $\lambda$ to the subsegment of the tie going from the boundary of $R(e)$ to $p$ following the tie in the positive direction.
				}
				\label{fig:deflation-between-rectangles}
			\end{minipage}
		\end{figure}
		
		We observe that the map $f_e$ is surjective as both $\O_\lambda(X)$ and $\lambda$ no not have atomic leaves: the first one by definition, the second one thanks to our choice of the representative $\hat\lambda$ for $\lambda$ in $\Gr(X,\lambda)$.
		We also notice that by invariance of the transverse measure upon transverse homotopy, each tie is mapped to a vertical segment in $Q(e) = [0,l]\times [0,h]$ and similarly each leaf of $\lambda$ is mapped to a horizontal segment. In other words, $\O_\lambda(X)$ and $\lambda$ are mapped to the vertical and horizontal measured foliation of the Euclidean rectangle.
		
		\begin{remark}
			This is possible as a measured geodesic lamination, is either a multicurve or has uncountably many leaves. In our case, since we replaced each atomic leaf with a foliated cylinder, $\lambda$ is such that in every rectangle $R(e)$ it has uncountably many leaves.
		\end{remark}
		
		We also observe that $f_e$ is not in general a homeomorphism as it is not necessarily injective. Indeed, if $\lambda$ is not a foliation in $R(e)$, then the restriction of $f_e$ to a tie is constant on each subsegment disjoint from $\lambda$.
		Indeed, along a tie the map $f_e$ behaves very similarly to a Cantor function (see Appendix \ref{sec:cantor}).
		
		\begin{remark}
			We notice, though, that when $\lambda$ has a closed leaf with positive mass, so $\Gr(X,\lambda)$ has an inserted cylinder, the map $f_e$ restricted to the intersection between $R_e$ and the cylinder is an isometry, as $\lambda$ and $\O_\lambda(X)$ already form a pair of orthogonal foliations on the cylinder. In particular, this shows that this construction extends the one we described above in the special case of $\lambda$ multicurve. 
		\end{remark}

		\paragraph*{Gluing the rectangles}
		Let us observe that the rectangles $R(e)$ are adjacent to each other along ties corresponding to switches of $\tau$. Let us consider the collection of Euclidean rectangles $Q(e)$ for all branches $e$ of $\tau$ and glue them along their vertical sides following the pattern induced by the adjacency of the respective rectangles $R(e)$. This means that for a segment $a$ of the tie separating $R(e)$ from $R(e')$, we glue $Q(e)$ to $Q(e')$ by attaching $f_e(a)$ to $f_{e'}(a)$ so that $f_e, f_{e'}$ patch together to form a continuous map from $R(e)\cup R(e')$ to $Q(e) \cup Q(e')$.
		
		This gluing is induced by piecewise isometries between boundaries of the rectangles $Q(e)$. Indeed, any subsegment of $a$, is mapped by $f_e$ and $f_{e'}$ to a segment of length equal to the transverse measure of $a$ given by $\lambda$.
		Since by construction the maps $(f_e)_e$ patch together nicely, we can call $f: N_\delta(\lambda) = \bigcup_e R_e \to \bigcup_e Q(e)$ the resulting continuous map, where we mean the rectangles $(Q(e))_e$ glued along the vertical sides as described.
		
		In order to get a half-translation surface, we just need to perform the gluing of the horizontal sides of the rectangles $Q(e)$'s. For this we proceed very similarly to the multicurve case above.
		Let us consider one connected component $A$ of the complement of $N_\delta(\lambda)$. The orthogeodesic foliation $\O_\lambda(X)$ restricted to $A$ is a singular foliation with leaves orthogonal to the boundary, as we showed before, so in particular transverse to it. 
		Exactly as in the multicurve case, the singular foliation induces a gluing pattern on the boundary of $A$.
		We now transfer this gluing pattern via the map $f$ to the Euclidean rectangles $(Q(e))_e$. If a segment $a$ on $\de A \subset \de N_\delta(\lambda)$ is paired by $\O_\lambda(X)$ to the segment $b$ also on $\de A$ via the bijection $\phi: a \to b$, then we glue $f(a)$ to $f(b)$ so that $f|_a = f|_b \circ \phi$. This gluing is induced by isometries on the boundary of $\bigcup_e Q(e)$ as $\phi$ preserves the transverse measure with respect to $\O_\lambda(X)$ and $f$, along the boundary of $N_\delta(\lambda)$ sends transverse measure to Euclidean length.
		
		The result of all the gluing operations of the rectangles $(Q(e))_e$ is a half-translation surface, because obtained by gluing rectangles along their boundaries via piecewise isometries, vertical sides with vertical sides and horizontal with horizontal ones, hence the vertical and horizontal foliations on each rectangle extend to two orthogonal singular foliations on the entire surface. 
		
		We can now easily define the deflation map $\D$.
		
		\begin{definition}
			On every rectangle $R(e)$ we define $\D$ to coincide with $f_e$. With this definition the map is continuous on $N_\delta(\lambda)$ as along each tie separating two adjacent rectangles $R(e)$ and $R(e')$, $f_e$ and $f_{e'}$ match by construction.
			On every connected component in the complement of $N_\delta(\lambda)$, we define $\D$ to send each - possibly singular - leaf $l$ to the same point, determined by the image of $\D$ on the endpoints of $l$, which we already defined as they lie on the boundary of $N_\delta(\lambda)$.
		\end{definition}

		It is easy to see that by construction the pair of foliation and lamination $(\O_\lambda(X),\lambda)$ are sent by $\D$ to the vertical and horizontal foliation on the resulting half-translation surface, which is then $\O(X, \lambda)$.

\section{Geometric control of deflation}		\label{sec:deflation}

	In this section we present the proof of the main theorems of the article.

	At the core of the arguments lies \thref{th:deflation-control}, estimating the constant of $\epsilon$-isometry for the deflation map $\D: \Gr(X,\lambda) \to \O(X,\lambda)$.
	For this scope we will choose a different renormalization of the metric. More precisely, we want to rescale both $\Gr(X,\lambda)$ and $\O(X,\lambda)$ with the same constant $k^{-1}$, so that $\O(X,\lambda)$ has unit area. This means that $k^2 = \ell_X(\lambda)$ (see equation \ref{eq:area-equal-length}).
	Thus, the grafted surface $\Gr(X,\lambda)$ rescaled by $k^2=\ell_X(\lambda)$ will have the flat part of unit area. We remind that for this section we will consider only surfaces grafted along a non-zero lamination, corresponding to the domain $\T(S) \times \ML(S)$ of the map $\O$.

	\begin{remark}
		This normalization, for large grafting, meaning when $\ell_X(\lambda)$ is large, is Gromov-Hausdorff close to the unit-area normalization where the rescaling factor was $(2\pi |\chi(S)| + \ell_X(\lambda))^{-1/2}$. Indeed, the Gromov-Hausdorff distance between the two copies of the grafted surface, with the two different renormalizations, is realized by the identity and is half of the difference of their diameters, which goes to zero as $\ell_X(\lambda)$ goes to infinity.
		Therefore, showing convergence to a half-translation surface in this renormalization is equivalent to using the unit-area one.
	\end{remark}

	\begin{remark}
		This normalization hints to the dual point of view we will discuss in Section \ref{sec:inflation}, where grafted surfaces can be viewed as unit area half-translation surfaces where some negatively parts has been inserted, and the deflation map just undoes this operation that we will call inflation, exactly as Thurston's collapsing map undoes grafting. In view of this, we can call the grafted and the negatively curved part of $\Gr(X,\lambda)$ as respectively the flat and the \emph{inflated region}.
	\end{remark}

	Let us then assume from now that all the grafted surfaces $\Gr(X,\lambda)$ we considered are rescaled to have the flat part of unit area, and the half-translation surfaces $\O(X,\lambda)$ also have unit area.

	\begin{theorem}	\thlabel{th:deflation-control}
		Let $\Gr(X,\lambda)$ be a grafted surface.
		Assume $\O(X,\lambda)$ has injectivity radius larger than $\epsilon$ and that $\ell_X(\lambda) > 1$.
		Then there exists $C$ depending only on $\epsilon$ and on the topology of the surface such that the deflation map $\D: \Gr(X,\lambda) \to \O(X,\lambda)$ is a $C \cdot \left(\ell_X(\lambda)\right)^{-1/2}$-isometry compatible with the markings. In particular
		$$ d_{GH}(\Gr(X,\lambda), \O(X,\lambda)) \leq C \cdot \left(\ell_X(\lambda)\right)^{-1/2}. $$ 
	\end{theorem}

	\subsection{Proof of the main result}


		In this section we prove fundamental properties of the deflation map $\D$. First we show that the deflation map, similarly to Thurston's collapsing map, preserves the metric in the grafted region while collapsing the inflated pieces. 
		This is clear in the case of a multicurve, as the combinatoric of the grafted surface is finite and the deflation map simply restricts to a local isometry in the flat part and collapses the inflated part along the orthogeodesic foliation.
		In the general case we have to make a more precise meaning of "preserves the metric", as the grafted region can have empty interior.
		In order to control the metric, we will use the length of geodesic arcs, as we did in the geometric control of grafting.
		
		\paragraph*{Construction of pull-back arcs}

		Given a geodesic arc in $\O(X,\lambda)$, its pull-back will be a corresponding arc $\beta$ in $\Gr(X,\lambda)$ that we will use to compare distances on the two surfaces. The construction of $\beta$, similarly to what we did in Section \ref{sec:degrafting}, is obtained essentially by taking the preimage via $\D$ of $\alpha$ with some extra care where $\alpha$ is parallel to the horizontal foliation.

		\begin{remark}
			Let us first observe that from the construction of $\D$ follows immediately that the preimage of a point $p$ is either
			\begin{itemize}
				\item a single point when $p$ does not lie on a horizontal separatrix;
				\item a segment of a leaf of the orthogeodesic foliation contained in a hyperbolic piece when $p$ lies on a horizontal separatrix;
				\item a union of those, when $p$ is a singularity.
			\end{itemize}
		\end{remark}

		\begin{lemma}		\thlabel{th:pull-back-arc}
			Given $\alpha:(0,1) \to \O(X,\lambda)$ a geodesic arc disjoint from singularities and transverse to the horizontal foliation, then its preimage $\D^{-1}(\alpha)$ is an arc in $\Gr(X,\lambda)$.
		\end{lemma}
		\begin{proof}
			It follows easily from the definition of $\D$, especially from the interpretation of the deflation viewed as collapsing the inflated pieces to the spine $\Sp$.

			We can observe that if $\lambda$ has closed leaves, then, as we showed, the cylinders corresponding to them are mapped isometrically by $\D$ to $\O(X,\lambda)$.
			Let us call $C$ the union of the interior of their images.
			Thus, the preimages of subsegments of $\alpha$ in $C$ are isometric copies of them in the grafted cylinders.

			\begin{figure}[ht]
				\includesvg[width=0.75\textwidth]{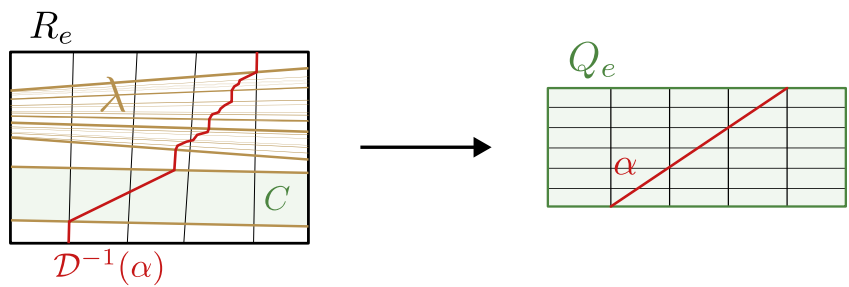}
				\centering
				\vspace{0.2cm}
				\begin{minipage}{0.8\textwidth}
					\caption{
						The preimage via $\D$ of a non-horizontal geodesic arc $\alpha$.
					}
					\label{fig:cantor-map}
				\end{minipage}
			\end{figure}

			Meanwhile, subsegments of $\alpha$ in the complement of $C$ have preimages that are arcs where each intersection point of $\alpha$ with $\D(\Sp)$ is inflated to a segment of the orthogeodesic foliation contained in the negatively curved part.

		\end{proof}

		Let us now define the pull-back of geodesic arcs more precisely.

		\begin{lemma}
			Given $\alpha:[0,1] \to \O(X,\lambda)$ a geodesic arc, there exists an arc $\beta$ in $\Gr(X,\lambda)$ such that  $\alpha = \D \circ \beta$ and such that the intersection $\beta_{infl}$ of $\beta$ with the interior of the inflated part is a union of segments of leaves of the orthogeodesic foliation of $\Gr(X,\lambda)$.
		\end{lemma}

		\begin{proof}[Proof of \thref{th:pull-back-arc}]
			The geodesic arc $\alpha$ can go through finitely many singularities, in order $v_1, \dots, v_{n-1}$, splitting $\alpha$ in a concatenation of finitely many consecutive geodesic segments $\alpha_1, \dots, \alpha_n$, each not containing a singularity in its interior. 
			Since $\D$ is a homotopy equivalence, there exists in $\D^{-1}(\alpha)$ an arc $\beta$ such that $\D \circ \beta = \alpha$. As it may be not unique, we will explicitly construct it
			piece by piece, that is we will choose first for the interior of each $\alpha_i$ an open arc $\beta_i$ such that $\alpha_i = \D \circ \beta_i$, then for every $i$ join the arcs $\beta_i$ to $\beta_{i+1}$ with an arc in $\D^{-1}(v_i)$.
			For simplicity of notation, from now on by $\alpha_i$ we mean its interior, so with endpoints excluded.
			Each segment $\alpha_i$ can be horizontal, or transverse to the horizontal foliation $\lambda$. We will deal with the two different cases separately, and then deal with the singularities.

			\emph{Horizontal segments.} If $\alpha_i$ is horizontal, either it lies on a non-singular leaf, in which case $\D|_{\alpha_i}$ is invertible and is an isometry, so there is only one possible choice for $\beta_i$; or $\alpha_i$ lies on a horizontal separatrix. In this last case then $\D^{-1}(\alpha_i)$ is a portion of a hyperbolic piece (regularly) foliated by the orthogeodesic foliation.
			In this case we choose $\beta_i$ to be on the boundary of such a region (see Figure \ref{fig:pull-back-arc}), that is on $\lambda$, so its length will be the same as the one of $\alpha_i$. Notice that we have two choices here. Both choices will equally work for our scopes. For coherence, we could also consistently always choose the one on the right (this is well-defined from the orientation of $\alpha$ and the orientation of the surface).
			
			\emph{Non-horizontal segments.} If $\alpha_i$ is not horizontal, which means it is transverse to the horizontal foliation, then we showed that its preimage $\D^{-1}(\alpha_i)$ is already an arc,
			so we simply choose $\beta_i$ to be $\D^{-1}(\alpha)$.
			
			\emph{Vertices.} The preimage $\D^{-1}(v_i)$ of a singular point $v_i$ of order $m$ is, by construction of the orthogeodesic foliation, the union of $m$ segments of the same length, sharing one endpoint.
			So for every singularity in the interior of $\alpha$, we can join the two arcs $\beta_i, \beta_{i+1}$ by following two of the $m$ segments in $\D^{-1}(v_i)$ as in Figure \ref{fig:pull-back-arc}. Which segments to follow depends on the choice for the horizontal segments we discussed above, but this will not have an impact on the length as the $m$ segments have all the same length.

			\begin{figure}[ht]
				\includesvg[width=0.9\textwidth]{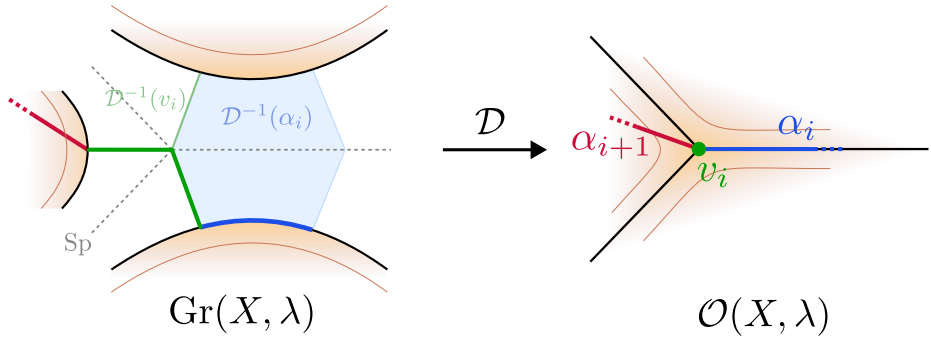}
				\centering
				\vspace{0.2cm}
				\begin{minipage}{0.8\textwidth}
					\caption{Summary of the construction of the pieces of the pull-back arc $\beta$: in green a vertex and the selected path in its preimage, and in blue a horizontal segment and the selected path as pull-back in its preimage, and in red a segment transverse to the horizontal foliation and its pull-back.}
					\label{fig:pull-back-arc}
				\end{minipage}
			\end{figure}

			This concludes the construction of the arc $\beta$ meeting the requirements.
		
		\end{proof}

		\begin{definition}
			With the notation above, we call $\beta$ the pull-back of $\alpha$ and write $\beta = \D^* \alpha$.
		\end{definition}

		The idea of the proof of the main theorem will be, as in the case of small grafting, to show good control of the length of $\beta$ in comparison to $\alpha$ in order to prove that $\D$ is an $\epsilon$-isometry.
		The argument is split in multiple lemmas that we prove here below, and combine at the end of the section to deduce the main theorem. 

		First, we show that, as we heuristically depicted, the increment of length from $\alpha$ to $\beta$ is given exactly by the length of the inflated part $\beta_{infl}$ of $\beta$.
		This is trivial in the case of a multi curve, but less clear for the general case.

		\begin{lemma}	\thlabel{th:infl-flat-decomp}
			Given $\alpha$ a geodesic arc in $\O(X,\lambda)$ and $\beta = \D^*\alpha$ its pull-back in $\Gr(X,\lambda)$,
			then 
			$$ \Len(\beta) = \Len(\alpha) + \Len(\beta_{infl}) .$$
		\end{lemma}
		\begin{proof}
			We first prove it in the case when $\lambda$ is a multicurve. In this case the grafted surface is a union of finitely many rescaled hyperbolic pieces and finitely many cylinders. So $\beta$ is a concatenation of finitely many arcs, each one contained either in the interior of the inflated part or in a cylinder, considered closed, meaning boundary included.
			The segments in the interior of the inflated part constitute $\beta_{infl}$ by definition, while 
			the arcs in the cylinders are mapped isometrically to $\O(X,\lambda)$, so the sum of their length is equal to $\Len(\alpha)$.
			
			The general case where $\lambda$ is any measured lamination will follow by a continuity argument.
			Notice first that in this general case we can still make sense of $\Len(\beta_{infl})$, as $\beta_{infl}$, in the general case, is a countable union of sub arcs of $\beta$.
			We cannot say the same for the complementary $\beta \setminus \beta_{infl}$, as this is a positive-measured Cantor set.
			We could define its length as its Lebesgue measure induced by the length along $\beta$, but then we would need to talk about measure preserving properties of $\D$.
			More simply we can define its length as $\Len(\beta) - \Len(\beta_{infl})$ which in the case when $\lambda$ is a multicurve, equals the length of the image $\D(\beta) = \alpha$.

			Let us then approximate $\lambda$ by a sequence of multicurves $\mu_n$ such that their support also converges in the Hausdorff sense to the support of $\lambda$ (see \thref{th:hausdorff-approx}).
			We know by \thref{th:grafting-continuity} that the metrics of the grafted surfaces $\Gr(X,\mu_n)$ converge to one of $\Gr(X,\lambda)$.

			By \thref{th:O-continuity} we have that $\O(X,\mu_n)$ converges to $\O(X,\lambda)$.
			More precisely for $n$ large enough, since the support of $\mu_n$ will lie close to the support of $\lambda$, the combinatorics of the construction of $\O(X,\mu_n)$ will be the same as the one for $\O(X,\lambda)$ and only the sizes of the Euclidean rectangles $(Q_e)_e$ will be different, but converging to the ones of the rectangles of $\O(X,\lambda)$.
			So for every $n$ large enough, we can identify $\O(X,\mu_n)$ topologically with the surface $\O(X,\lambda)$ but endowed with a slightly different metric.
			In this way we can see $\alpha$ as an arc on all the half-translation surfaces $\O(X,\mu_n)$. Let us call $\D_n: \Gr(X,\mu_n) \to \O(X,\mu_n)$ the deflation maps. We now produce for each of them, $\beta_n = \D_n^*\alpha_n$ the pull-back arc on $\Gr(X,\mu_n)$, which we all see as the same smooth surface with different metric.

			One can easily see that the deflation maps $\D_n$ also converge, and so do the arcs $\beta_n$ (up to making consistent choices in their construction as we mentioned). From which we have convergence for the lengths of $\beta_n$.

			Finally, one can also see that the interior of the region of negative curvature with respect to the metric of $\Gr(X,\mu_n)$ converges to the one with respect to $\Gr(X,\lambda)$. In combination with the fact that $\beta_{infl}$ is always transverse to the boundary of the inflated region (as it follows the orthogeodesic foliation) we have convergence of its length.

			So we can conclude that the equation passes to the limit.
			
		\end{proof}

		\begin{remark}
			The construction is stable under approximation of $\lambda$ with a sequence of weighted multicurves $\mu_n \to \lambda$ with $\supp(\mu_n) \to \supp(\lambda)$.
			As pointed out in \cite[Construction 5.6]{CF21}, it is enough to choose $\epsilon$ small enough so that $N_\epsilon(\lambda)$ is topologically stable under small variations of $\epsilon$. So for $n$ large enough, by the Hausdorff convergence of the supports of $\mu_n$, $N_\epsilon(\mu_n)$ will also have the same topology, and so will give topologically the same rectangle decomposition.

			The map $\D$ depends then continuously on the geometric data  
			Thus by the convergence of the measures of $\mu_n$ and the orthogeodesic ...
			$\D_n$ converges to the map $\D$. 
		\end{remark}

		\begin{lemma} 	\thlabel{th:deflation-lipschitz}
		 The deflation map $\D$ is 1-Lipschitz.
		\end{lemma}
	\begin{proof}
		Let us start with the case when $\lambda$ is a weighted multicurve.
		Then the restriction of $\D$ to the grafted cylinders is a local isometry, while the restriction to a negatively curved part is a contraction for the following reason.
		From the construction, we can see $\D$ locally as the composition of the closest point projection to the lamination, which in negative curvature is a contraction, which is then mapped isometrically into $\O(X,\lambda)$.
		
		In the general case, it follows from the stability for approximation by multicurves, and continuity of the of $\D$ from the grafting data.
	\end{proof}

		Since what we will want to show is that the metric of the grafted surface converges to the one of the half-translation surface, what we aim to prove is an upper bound for $\Len(\beta_{infl})$.
		The first key ingredient is to show that under certain conditions the inflated part is uniformly slim.

		Let us remind that the inflated parts have negative constant curvature $-k^2$. We remind the normalization we are currently using is the one for which the flat part of $\Gr(X,\lambda)$ has area 1, so the multiplicative scaling factor is $k^{-1} = \ell_X(\lambda)^{-1/2}$.

		\begin{proposition}		\thlabel{th:slim-inflation}
			Let $\Gr(X,\lambda)$ be a grafted surface with $\ell_X(\lambda)>1$ and injectivity radius larger than $\epsilon$.
			Then there exists a constant $D$, depending only on $\epsilon$, such that the inflated part in $\Gr(X,\lambda)$ is $D/k$-slim, meaning that any point in the interior of the inflated part, lies at distance at most $D/k$ from its boundary.
		\end{proposition}
		\begin{proof}
			The hyperbolic surface $X$ might not be $\epsilon$-thick in general, so let us consider its $\epsilon$-thick-thin decomposition.
			If it is trivial, that is $X$ is $\epsilon$-thick, then there is an upper bound $D$ for its diameter, by compactness of the $\epsilon$-thick part of the moduli space $\M(S)$. Such $D$ will trivially satisfy our proposition in this case, as the inflated part is obtained by rescaling $X$ by $k^{-1}$ and cutting it along $\lambda$, resulting in pieces $D/k$-slim.

			In the general case, assuming without loss of generality $\epsilon$ small enough, the thin components are tubes.
			Let us then consider a thin component of $X$, meaning an annulus where every point has injectivity less than $\epsilon$ with respect to the hyperbolic metric of $X$, in particular the shortest closed geodesic $\gamma$ in such annulus has length less than $2\epsilon$.
			Then there must be a leaf $l$ of $\lambda$ crossing $\gamma$, otherwise to $\gamma$ corresponds an even shorter curve in $\Gr(X,\lambda)$, as the grafting does not affect $\gamma$ and the rescaling is by $k^{-1} = \ell_X(\lambda) < 1$ by assumption. The leaf $l$, since it intersects $\gamma$, by simple hyperbolic geometry, traverses the thin tube in its entire length from a boundary component to the other.
			From this follows also that every thick component of $X$ intersects at least one leaf of $\lambda$ since, by definition, every thick component is adjacent to at least a thin component.
			
			It is now easy to show the slimness of the inflated part in $\Gr(X,\lambda)$. 
			In $X$, any point in a thin component, will lie at distance at most $\epsilon$ from the leaf of $\lambda$ traversing it. On the other side, the $\epsilon$-thick components have uniformly upper bounded diameter $D=D(\epsilon, g)$ depending only on $\epsilon$ and on the genus of the surface. We can assume $D>\epsilon$.
			Since we showed that every thick component intersects at least one leaf of $\lambda$, then any point in a thick component is at distance at most $D$ from $\lambda$.
			In $\Gr(X,\lambda)$ the metric is rescaled by a factor $1/k$ with respect to the one coming from the hyperbolic surface $X$, so any point will be distant at most $D/k$ from the boundary of the inflated part. 
		\end{proof}

		The lemma above is nevertheless not sufficient to provide a bound for the length of $\beta$ gained in the inflated parts, as in general $\beta_{infl}$ will be a union of countably many segments. For this reason we need the following lemmas, that are essentially the analogue of \thref{th:intersection-bound}.

		\begin{lemma}		\thlabel{th:inflated-part-bound}
			Consider $\D: \Gr(X,\lambda) \to \O(X,\lambda)$ as above, and assume $\O(X,\lambda)$ has injectivity radius larger than $\epsilon$.
			If $\alpha$ is a non-horizontal geodesic arc in $\O(X,\lambda)$ disjoint from singularities, and it is distance minimizing, then $\beta = \D^*(\alpha)$ is such that 
			$$ \Len(\beta_{infl}) \leq C_{\epsilon} k^{-2} $$
			where the constant $C_\epsilon$ depends only on $\epsilon$.
		\end{lemma}
		\begin{proof}

			As we anticipated, the idea for the argument is the same as in \thref{th:intersection-bound}:
			each segment of $\beta_{infl}$ has a fairly sized neighbourhood in the inflated part, but the latter has total area $2\pi|\chi(S)|/k^{2}$, which is going to 0 with $k$ going to infinity.
			
			Let us consider for every point $\alpha(t)$ in $\alpha \cap \Sp$, the horizontal segment $U_t$ centred at $\alpha(t)$, and which in normal coordinates has endpoints $\alpha(t)+i \epsilon/2$ and $\alpha(t) - i\epsilon/2$.
			Note that it does not intersect itself as
			we assumed that $\O(X,\lambda)$ has injectivity radius at least $\epsilon > \epsilon/2$.
			
			Moreover, all the segments $U_t$ will be disjoint.
			Suppose by contradiction that $U_t$ intersects $U_{t'}$, with $t\neq t'$. Then we can connect $\alpha(t)$ and $\alpha(t')$ with a horizontal segment of length less than $\epsilon$. Again by the assumption on the injectivity radius, such horizontal segment is length minimizing between $\alpha(t)$ and $\alpha(t')$. But being $\alpha$ length minimizing, it implies that that segment lies in $\alpha$, which contradicts the assumption $\alpha$ transverse to the horizontal foliation.

			\begin{figure}[ht]
				\includesvg[width=0.65\textwidth]{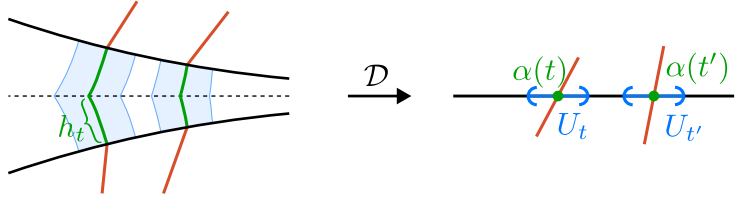}
				\centering
				\vspace{0.2cm}
				\begin{minipage}{0.8\textwidth}
					\caption{We call $h_t$ half the length of $\D^{-1}(\alpha(t))$.}
					\label{fig:inflation-area-gain}
				\end{minipage}
			\end{figure}

			Consider now $U$ the subset of $\lambda$ formed by all the (disjoint) segments $U_t$ we described
			$$ U := \bigcup_{\alpha(t) \in \Sp} U_t .$$
			
			Similarly to what we did for grafting, we observe that $\D^{-1}(U)$, being contained in the inflated part of $\Gr(X,\lambda)$, has area bounded above by the total area of the inflated part, which is $2\pi|\chi(S)|k^{-2}$.
			On the other hand, we also have a lower bound for the area of $U$.
			Indeed, $U$ is the disjoint union of $\D^{-1}(U_t)$ for countably many values of $t$.
			Each region $\D^{-1}(U_t)$ is made up of two isometric pieces, whose area we can lower bound thanks to \thref{th:hyp-trapezium-area}
			$$ \Area(U) = \sum_{t} \Area(\D^{-1}(U_t)) \geq \sum_{t} 2 \epsilon h_t = 2\epsilon \Len(\beta_{infl}) .$$
			Combining the two area bounds we obtain
			$$\Len(\beta_{infl}) \leq \frac{\pi|\chi(S)|}{\epsilon} k^{-2} .$$
		\end{proof}

		Here is the hyperbolic geometry lemma about area estimate we used in the argument above. Essentially it estimates the area of a quadrilateral in $\H^2$ with two adjacent right angles by comparison with an analogue Euclidean counterpart.

		\begin{lemma}		\thlabel{th:hyp-trapezium-area}
			Let $r$ be a line in $\H^2$ parametrized by arc-length and $\eta$ be the foliation whose leaves $(\eta_t)_{t\in\R}$ are such that $\eta_t$ is the geodesic orthogonal to $r$ in $r(t)$. Let $l$ be any other line disjoint from $r$. Denote $r(t_0)$ a point on $r$, and $U$ be the region delimited by $r,l, \eta_{t_0+\delta}, \eta_{t_0-\delta}$ as in Figure \ref{fig:trapeziums}.
			Call $h(t)$ the length of the segment of $\eta_{t}$ between $r$ and $l$, then
			$$\Area(U) \geq 2\delta h(t_0) .$$
		\end{lemma}

		\begin{proof}
			Consider the coordinates $\phi: \R^2 \to \H^2$ that maps the $x$ axis isometrically to $r$, so $\phi(t,0) = r(t)$ and sends each horizontal line $\{x=t\}$ isometrically to $\eta_t$. Assume also that $l \subset \phi(\{y>0\})$.

			\begin{figure}[ht]
				\includesvg[width=0.9\textwidth]{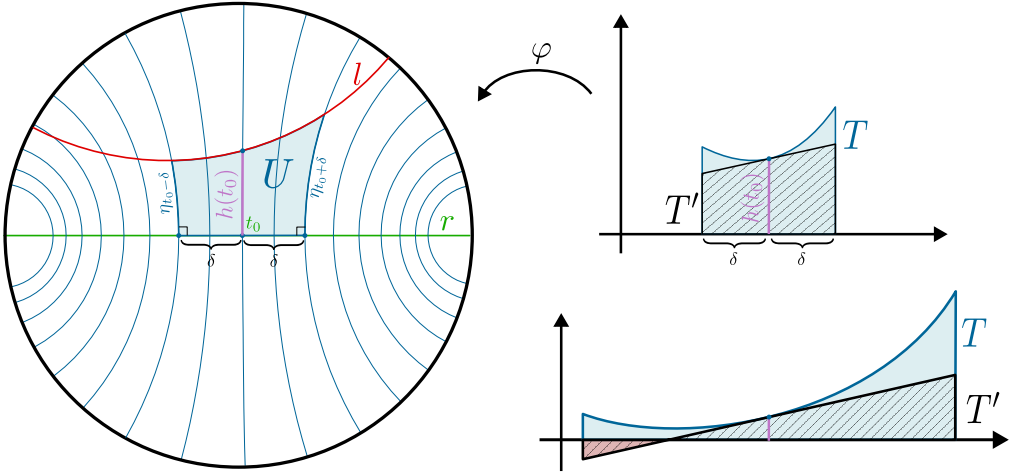}
				\centering
				\vspace{0.2cm}
				\begin{minipage}{0.8\textwidth}
					\caption{Left: the hyperbolic quadrilateral $U$. Top right: the Euclidean chart. Bottom right: the case when $T'$ is not simple.}
					\label{fig:trapeziums}
				\end{minipage}
			\end{figure}
			
			In this setting $l$ is the image under $\phi$ of the graph of the function $h$ and $U$ is the image under $\phi$ of the region $T$ under the graph of $h$ restricted to $(t_0-\delta, t_0+\delta)$.
			Because of the negative curvature of $\H^2$, the map $\phi$ is a dilation from the Euclidean metric on $\R^2$ to the hyperbolic metric on $\H^2$, so $\Area(U) \geq \Area(T)$.
			Always due to negative curvature, we know that the function $h$ is convex (it is indeed a hyperbolic cosine or an exponential), so $\Area(T) \geq \Area(T')$ where $T'$ is the trapezium under the tangent in $t_0$ to the graph of $h$, hence
			$$ \Area(U) \geq \Area(T) \geq \Area(T') = 2\delta h(t_0) .$$
			Note that even in the case when the tangent line delimiting $T'$ intersects the $x$ axis inside the interval $(t_0-\delta, t_0+\delta)$, then one can just consider only the part above the $x$-axis of $T'$, whose area will be even larger than $2\delta h(t_0)$, as that is the signed area of $T'$, where the part below the $x$-axis counts as negative.
		\end{proof}
		
		\begin{remark}
			Note that the previous lemma holds true even if we replace $\H^2$ with a rescaled version of it, meaning with constant curvature $-k^2$. Indeed, we just used that in negative curvature the parametrization $\phi$ is a dilation. Or one can just rescale by $k$ the initial data of the problem, and reach the same conclusion, as the inequality would be homogeneous in the scale factor $k$.
		\end{remark}

		Finally, we can combine all the previous lemmas to compose a proof of \thref{th:deflation-control}.

		\begin{proof}[Proof of \thref{th:deflation-control}]
			Let us consider the deflation map $\D: \Gr(X,\lambda) \to \O(X,\lambda)$. 
			We know that $\D$ is 1-Lipschitz.
			We want to show that $\D$ also does not shrink distances by more than an additive error proportional to $k^{-1}$, and then conclude by \thref{th:GH-criterion}. 

			For any two points $x,y$ in $\Gr(X,\lambda)$ we want to show that 
			\begin{equation}	\label{eq:defl-shrink-control}
				d(\D(x), \D(y)) \geq  d(x,y) - Ck^{-1}. 
			\end{equation}
			
			Let us consider the geodesic arc $\alpha$ joining $\D(x), \D(y)$ and realizing their distance on $\O(X,\lambda)$.
			Let us consider the pull-back arc $\beta = \D^*\alpha$. The endpoints of $\beta$ are not necessarily $x,y$. They lie in $\D^{-1}(\D(x))$ and in $\D^{-1}(\D(y))$, which as we said above, can be either a point, a segment of a leaf, or a star-shaped union of singular leaves. In any case, by \thref{th:slim-inflation}, the length of such segments is at most $Dk^{-1}$, so we can extend $\beta$ by at most  $4Dk^{-1}$ so that it joins $x$ to $y$. This contributes to $\Len(\beta_{infl})$.
			
			Now the length of $\beta$ gives an upper bound for the distance between $x$ and $y$. By \thref{th:infl-flat-decomp} we have indeed
			$$ d(\D(x),\D(y)) = \Len(\alpha) = \Len (\beta) - \Len(\beta_{infl}) \geq  d(x,y)-  \Len(\beta_{infl}) $$
			so all we need is an upper bound for $ \Len(\beta_{infl})$.
			By the construction of $\beta$, the inflated part $\beta_{infl}$ only comes from the preimage of segments of $\alpha$ transverse to the horizontal foliation, and the preimage of singularities crossed by $\alpha$. The inflated part coming from a segment transverse to the horizontal foliation, by \thref{th:inflated-part-bound}, is at most $C_\epsilon k^{-2}$.
			The one from a singularity is, by \thref{th:slim-inflation} at most $2Dk^{-1}$. The number of such segments and singularities is bounded, as $\alpha$ being minimizing will not cross any singularity twice and the total number of singularities is at most $4g-4$.
			Summing all the contributions to $\Len(\beta_{infl})$ we obtain, with the assumption $k>1$,
			\begin{align*}
				\Len(\beta_{infl})
				&\leq 4Dk^{-1} + (4g-4)2Dk^{-1} + (4g-3)C_\epsilon k^{-2} \\
				&\leq (8g-4)Dk^{-1} + (4g-3)C_\epsilon k^{-2} \\
				&\leq \left[(8g-4)D + (4g-3)C_\epsilon \right]k^{-1} \\
				&\leq (4g-2)(2D + C_\epsilon) k^{-1}
			\end{align*}
			This concludes the proof of \ref{eq:defl-shrink-control} with constant $C = (4g-2)(2D + C_\epsilon)$. So $\D$ is a $Ck^{-1}$-isometry and, as anticipated, by \thref{th:GH-criterion} follows the thesis, where we remind that $k^2 = \ell_X(\lambda)$ so $k^{-1} = \ell_X(\lambda)^{-1/2}$.
			
		\end{proof}

		We now have all the tools and the language to show how the main theorem can be deduced from \thref{th:deflation-control}.

		\begin{proof}[Proof of \thref{th:bordification}]
			Let $q \in \P_\C\QT(S) \subset \P\Met(S)$ be a unit area half-translation surface.
			Then by \thref{th:orthogeodesic}, $\O$ is a bijection, then we can find a sequence of grafted surfaces $\O^{-1}(nq)$ that deflate to $q$ and for which $\ell_X(\lambda)$ is going to infinity. Then by \thref{th:deflation-control} such a sequence converges to $q$.
			See Section \ref{sec:inflation-rays} for a more detailed discussion of such sequences of grafted surfaces.
		\end{proof}

\section{Inflation}		\label{sec:inflation}

	We introduced a bordification $\bar{\PT(S)}$ of the space of grafted surfaces $\PT(S)$. Such bordification encompasses all the hyperbolic surfaces, as contains $\T(S)$, the unit-area half-translation surfaces as it contains $\P_\C\QT(S)$, and in between those, grafted surfaces that are partly hyperbolic and partly flat.
	Grafting allows to see this third family of surfaces as deformation of hyperbolic ones.
	In this section we will describe a complementary point of view in which they are seen as deformation of the flat half-translation surfaces. We will call such deformation \emph{inflation}, as it goes in the opposite direction of the deflation map.
	The ingredients for this description are all already discussed in the previous sections; this section introduces only a different perspective on the same phenomenon.
	
	\subsection{Inflation parametrization}
		
	In this section we will define the inflation map, and rewrite the main result in terms of this.
	As we said, the inflation is meant to go in the opposite direction of the deflation, so given a half-translation surface $q$, it produces a surface whose deflation is $q$. 

	Let us start then defining the inflation as the map $\O^{-1}$
	$$\Infl = \O^{-1}: \QT(S) \to \PT(S).$$
	
	Given $q\in \QT(S)$ we think of it as the datum of a marked unit area half-translation surface $q'=\frac{q}{\sqrt{|q|}}$ together with a direction given by the foliation $\re(q)$, and a positive real parameter $k$ given by $k^2 = |q|$. So we have that $q=kq'$ with $q'$ unit area.
	While the grafting datum is encoded by a pair made of the surface and the grafting lamination, here everything can be encoded in a single quadratic differential $q$.
	
	So $\Infl(q) = \Infl(kq')$ is a grafted surface, say $\Gr(X,\lambda)$, then this means that 
	$$ \O_{\lambda}(X) = \re(q) = \re(kq') = k\re(q') \qquad \text{and} \qquad 
	\lambda = \im(q) = \im(kq') = k \im(q') . $$
	
	Moreover, we remind that the transverse measure to the foliation $\O_\lambda(X)$ is given by the $X$-hyperbolic length along leaves of $\lambda$. So we have that the following holds.
	
	\begin{equation} \label{eq:area-equal-length}
		\ell_X(\lambda) = i(\O_\lambda(X), \lambda) = i( \re(q), \im(q) ) = |q| = k^2
	\end{equation}
	where we used that the intersection form valued on vertical and horizontal foliations gives the Euclidean area $|q|$ of the half-translation surface.
	As a consequence, \thref{th:deflation-control} rewrites in terms of inflation as follows.
	\begin{theorem} \thlabel{th:infl-to-zero}
		Let $q$ be a unit quadratic differential, $k>1$ and let us consider $\epsilon$ smaller than the injectivity radius of $q$.
		Then the deflation map $\D: \Infl(kq) \to q$ is a $C/k$-quasi isometry, where the constant $C$ depends only on $\epsilon$ and on the genus of the surface. In particular, we have that 
		$$d_{GH}(\Infl(kq), q) \leq C / k .$$
		We remind that in this context $\Infl(kq)$ is considered rescaled such that the area of the grafted part is 1, and by $q$ we mean the associated unit area half-translation surface.
	\end{theorem}

	We now remind that $\QT(S)$ is a complex vector bundle over $\T(S)$, where every fibre $Q(X)$ with $X \in \T(S)$ is the vector space of holomorphic quadratic differentials on the Riemann surface $X$. Each fibre has a projective compactification given by $\bar{Q(X)} = Q(X) \cup \P Q(X)$ where $\P_\C Q(X) = (Q(X) \setminus \{0\}) / \C^*$ is the complex projectification of $Q(X)$.
	Applying this compactification fibre-wise we obtain the bordification $\bar{ \QT(S)} =  \QT(S) \cup \P_\C\QT(S)$.

	We now want to extend the inflation $\Infl: \QT(S) \to \PT(S)$ to the bordifications introduced above and obtain a mapping
	$$ \Infl: \bar{ \QT(S)} \to \bar{\PT(S)} .$$
	In order to do that, we notice that by definition both boundaries are described by $\P_\C\QT(S)$. We define then the inflation map $\Infl$ to be the identity on the boundary.
	As pointed out already $\Infl$ is not continuos as the map $\O^{-1}$ is not globally continuous.
	Our main result is equivalent in this setting to the continuity of $\Infl$ at the boundary. 
	From now on we indicate with $\T(S) \subset \bar{\PT(S)}$ the subset corresponding via grafting to $\T(S) \times \{0\} \subset \T(S) \times \ML_0(S)$.

	\begin{corollary}		\thlabel{th:inflation}
		The inflation map $\Infl: \bar{ \QT(S)} \to \bar{\PT(S)}$ is a bijection between the domain and the complement of $\T(S)$ in $\bar{\PT(S)}$, and it is continuous on the boundary $\P_\C\QT(S)$.
	\end{corollary}
	\begin{proof}
		It is easy to see that it is a bijection because we defined the inflation on $\QT(S)$ as $\O^{-1}$, which is a bijection between $\QT(S)$ and $\PT(S) \setminus \T(S)$. Then we extended it as the identity between the boundaries, which is again clearly bijective.
		So also the whole map $\Infl$ is bijective.

		Let us now check the continuity at the boundary, meaning that for every sequence $(q_n)_n $ converging projectively to $[q] \in \P_\C\QT(S)$, where $q$ is the unit quadratic differential representing the class $[q]$, then the surfaces $\Infl(q_n)$ converge in $\P\Met(S)$ to $\Infl([q]) = q$.

		The convergence of $q_n$ to $[q]$ implies that $|q_n|$ goes to infinity and that $q_n$ converges to $q$ also in $\bar{\PT(S)}$.
		In particular the injectivity radius of $q_n$ also converges to the injectivity radius of $q$, and is then uniformly lower bounded by a positive constant. 
		Then \thref{th:infl-to-zero} tells us that the deflation maps $\D_n: \Infl(q_n) \to q_n$ are $\epsilon_n$-isometries with $\epsilon_n$ going to zero as $n$ tends to infinity, because we showed that $|q|$ is going to infinity.

		Moreover, by definition of convergence in $\P\Met(S)$ we also have a sequence of homeomorphisms $\phi_n: q \to q_n$ compatible with markings that are $\epsilon_n$-bilipschitz on progressively larger subsets of $q$.
		One can easily verify then that the compositions $\phi_n^{-1} \circ \D_n: \Gr(X_n,\lambda_n) \to q$ satisfy the hypothesis for \thref{th:hausdorff-to-lipschitz} which then implies convergence in $\P\Met(S)$ as we wanted.

	\end{proof}

	To sum up we have that the space $\bar{\PT(S)}$ has two partial parametrization, given by grafting and by inflation.
	For the first one $\Gr: \T(S) \times \ML_0(S) \to \bar{\PT(S)}$ we have that
	\begin{itemize}
		\item $\Gr$ is a homeomorphism onto the complement in $\bar{\PT(S)}$ of the subspace $\P_\C\QT(S)$ corresponding to half-translation surfaces;
		\item to points of $\T(S)\times\{0\}$ correspond hyperbolic surfaces;
		\item to points of $\T(S)\times\ML(S)$ correspond properly grafted surfaces.
	\end{itemize}
	For the second one $\Infl: \bar{\QT(S)} \to \bar{\PT(S)}$ we have that
	\begin{itemize}
		\item $\Infl$ is a bijection between $\bar{\QT(S)}$ and the complement in $\bar{\PT(S)}$ of the space of hyperbolic surfaces $\T(S)$; it is not globally continuous, but it is continuous at the boundary and on the continuity regions of $\O^{-1}$ (see \thref{th:O-continuity})
		\item to points of $\P_\C\QT(S)$ correspond unit area half-translation surfaces;
		\item to points of $\QT(S)$ correspond properly grafted surfaces. 
	\end{itemize}

	\subsection{Inflation rays}		\label{sec:inflation-rays}

	In this section we want to summarize the geometric intuitive interpretation of inflation and describe inflation rays.

	As in grafting we have as model the simplest case of insertion of a single cylinder, let us try to describe more concretely the easiest example for inflation.
	
	\begin{example}
		Let us consider a unit quadratic differential $q$ where the diagram of horizontal separatrices is a (finite) connected graph $G$. From its embedding in $q$, $G$ also inherits the structure of a metric ribbon graph.
		Then from the previous section and the explicit construction of the deflation map, follows that $\Infl(kq)$ can be obtained by cutting open the half-translation surface $q$ along $G$ and gluing in its place a negatively curved surface $A_k$ with boundary, obtained as follows.

		Let us consider the $kG$ the ribbon graph $G$ with the metric rescaled by $k$. Then to it corresponds a hyperbolic surface $B_k$ such that the spine $\Sp(B_k)$ with the metric induced by the orthogeodesic foliation with respect to the border $\de B_k$ is $kG$. The surface $B_k$ exists and is unique, as discussed in \ref{sec:crowned-surfs}. The surface $A_k$ taken as $B_k$ rescaled by the constant $k^{-1}$.
		
		By \thref{th:deflation-lipschitz} we have that the injectivity radius of $\Infl(kq)$ is larger than the one of $q$, so \thref{th:slim-inflation} tells us that $A_k$ is $C/k$-slim for a constant $C$ depending only on  the injectivity radius of $q$.
		This implies that for $k$ going to 0, $A_k$ will shrink and deformation retract onto its spine, converging then to the metric graph $G$.
		The convergence of such sequences of rescaled hyperbolic surfaces with boundary is also studied in \cite{xu19}.
		
		To sum up, we can imagine inflation as taking the flat half-translation surface $q$ and inflating the graph $G$ of horizontal separatrices to a fat graph endowed with a metric with curvature $-k^2$ and geodesic boundary.
		
	\end{example}
	
	\begin{definition}
		Similarly to grafting rays $t \mapsto \Gr(X,t\lambda)$, we call \emph{inflation rays} the paths defined as $t \mapsto \Infl(t^{-1}q)$ where $q$ is a unit area quadratic differential.
	\end{definition}
	
	Note that we used the inverse of $t$ in the definition in order to have the following.

	\begin{proposition}
		The inflation ray $t \mapsto \Infl(t^{-1}q)$ is continuous and $\Infl(t^{-1}q)$ converges geometrically to $q$ as $t$ tends to 0.
	\end{proposition}
	\begin{proof}
		The inflation ray is continuous for $t \in (0,+\infty)$ since $\O^{-1}$ is continuous restricted to $\{ t^{-1}q \}_{t\in (0,+\infty)}$ by \thref{th:O-continuity} as the horizontal foliation of $t^{-1}q$ is only rescaled upon change of $t$, so its support does not change with $t$.
		The geometric convergence for $t$ approaching 0 follows from \thref{th:infl-to-zero}.
	\end{proof}
	
	As we have seen that a grafting ray is mapped through $\O$ to a rescaled Teichmüller line, we now want to see how an inflation ray is described in terms of grafting coordinates that is as $\Infl(t^{-1}q) = \Gr(X(t), \lambda(t))$ for suitable functions $X(t), \lambda(t)$.	
	Let us call $\eta = \re(q)$ and $\lambda = \lambda(1) = \im(q)$, then $\lambda(t) = t^{-1}\lambda$. Then we also remind that the map $\O_\lambda$ depends only on the support of $\lambda$, so $\O_{\lambda(t)} = \O_\lambda $, thus we have $t^{-1}\eta = \O_{\lambda(t)}(X(t)) =\O_{\lambda}(X(t)) $.

	\begin{corollary}
		The inflation ray $t \mapsto \Infl(t^{-1}q)$ coincides for $t>0$ with the path $\Gr(X(t), t^{-1}\lambda)$ where
		$X(t)$ is the generalized stretch ray defined by $X(s) = \O_\lambda^{-1}(s^{-1}\eta)$.
	\end{corollary}

	\begin{appendices}
		\section{One-dimensional grafting along a Cantor set}
		\label{sec:cantor}
		
		The scope of this appendix section is to give a more explicit description of a one-dimensional model for grafting in the case when the lamination is generic, so without closed leaves, since there is no straight forward cut-and-paste operation to describe the grafted metric.
		This one-dimensional model can be thought as what happens to a smooth arc on a hyperbolic surface $X$, orthogonal to a measured lamination $\lambda$, when the surface is grafted along $\lambda$. The analogies with this are the ones suggested by the choice of notation.

		Let $L\subset (0,1)$ be a Cantor set, meaning closed and with empty interior and Lebesgue measure 0, but of uncountable cardinality, like for example the very well known ternary Cantor set.
		Its complement $L^c$ is open, and in particular can be written as a countable union of disjoint intervals.
		$$ U = \bigcup_{n\in\N} (a_n, b_n) $$
		Let us now consider any finite measure $\lambda$ in $(0,1)$ concentrated on $L$. The function 
		$$ f: (0,1) \to \R; \quad f(x) = \lambda((0,x)) $$
		is non-decreasing, and constant on $(a_n,b_n)$.
		Let us now consider the set obtained from $U$ by translating each interval $(a_n,b_n)$ according to $f$ as follows.
		$$ U' = \bigcup_{n\in\N} (a_n+f(a_n), b_n+f(a_n)) \subset (0,1+|\lambda|) $$
		Notice that the intervals in the union are disjoint because $(a_n,b_n)$ are and $f$ is non-decreasing, so each pair of intervals can only get further away from each other.
		By sigma additivity and translation invariance of the Lebesgue measure, we have that $|U'| = 1$. As a consequence, its complement $\hat L$ in $(0,1+|\lambda|)$ is a closed set with measure $|\lambda|$.
		If the measure $\lambda$ has no atoms, then $\hat L$ also has empty interior, but, as we showed, positive Lebesgue measure. Such sets are also known as Smith-Volterra-Cantor sets, or fat Cantor sets.

		What follows can also be generalized in the case where $\lambda$ is allowed to have atoms. For sake of brevity of the exposition we stick to the generic case with no atoms, so that $f$ is continuous and makes the exposition easier.

		Let us consider the function $(0,1) \to (0,1+|\lambda|)$ given by $x \mapsto x+f(x)$. It is monotonic increasing and bijective. Let us call its inverse $\kappa: (0,1+|\lambda|) \to (0,1)$.
		\begin{proposition}
			 The function $\kappa$ maps isometrically the interval $(a_n+f(a_n), b_n+f(a_n))$ to $(a_n,b_n)$ for every $n\in\N$. This is the analogue of Thurston's collapsing map of \ref{sec:thurston-param}.
		\end{proposition}

		Let us consider now the composition $\D: f\circ\kappa: (0,1+|\lambda|) \to (0,|\lambda|)$. 
		\begin{proposition}
			The function $\D$ is constant on each connected component of the complement of $\hat L$, meaning on each interval $(a_n+f(a_n), b_n+f(a_n))$, and its restriction to the fat Cantor set $\hat L$ is surjective onto $(0,|\lambda|)$ and Lebesgue measure preserving.
		\end{proposition}

		\end{appendices}

\newpage
\printbibliography[heading=bibintoc]

\noindent
Andrea Egidio Monti\\
Max Planck Institut für Mathematik, Vivatsgasse 7, 53111 Bonn, Germany\\
e-mail: amonti@math.uni-bonn.de

\end{document}